\documentclass[10pt, reqno]{amsart}
\usepackage[T1]{fontenc}
\usepackage[utf8]{inputenc}
\usepackage{amsmath,amsfonts,amsthm,amssymb,amsxtra}
\usepackage{hyperref}
\usepackage{bbm} 
\usepackage{color}
\usepackage[margin=3cm]{geometry}
\usepackage{enumitem}
\setlist{  
listparindent=\parindent,
parsep=0pt,
}

\usepackage{comment}

\newtheorem{theorem}{Theorem}[section]

\newtheorem{lemma}[theorem]{Lemma}

\theoremstyle{definition}

\theoremstyle{remark}
\newtheorem{remark}[theorem]{Remark}

\numberwithin{equation}{section}

\newcommand{\eps}{\ensuremath{\varepsilon}}
\newcommand{\emptyparam}{\ensuremath{\,\cdot\,}}
\newcommand{\norm}[1]{\left\lVert#1\right\rVert}

\renewcommand{\epsilon}{\varepsilon}

\newcommand{\xl}{{x, \lambda}}

\newcommand{\pxi}{\partial_{x_i}}

\newcommand{\pl}{\partial_{\lambda}}

\newcommand{\Ds}{{(-\Delta)^{s}}}
\newcommand{\ds}{{(-\Delta)^{s/2}}}

\newcommand{\dd}{\, \mathrm{d}}
\newcommand{\N}{\ensuremath{\mathbb N}} 
\newcommand{\R}{\ensuremath{\mathbb R}} 

\renewcommand{\S}{\ensuremath{\mathbb S}} 
\newcommand{\scalprod}[1]{\ensuremath{\langle #1\rangle}} 

\DeclareMathOperator{\dist}{dist}

\begin{document}

\title[Stability of critical points of the fractional Sobolev inequality]{Stability with explicit constants of the critical points of the fractional Sobolev inequality and applications to fast diffusion}


\author[N. De Nitti]{Nicola De Nitti}
\address[N. De Nitti]{Friedrich-Alexander-Universität Erlangen-Nürnberg, Department of Mathematics, Chair for Dynamics, Control, Machine Learning and Numerics (Alexander von Humboldt Professorship), Cauerstr. 11, 91058 Erlangen, Germany.}
\email{nicola.de.nitti@fau.de}

\author[T. König]{Tobias König}
\address[T. König]{Goethe-Universität Frankfurt, Institut für Mathematik, 
Robert-Mayer-Str. 10, 60325 Frankfurt am Main, Germany.}
\email{koenig@mathematik.uni-frankfurt.de}

\thanks{}

\begin{abstract}
We study the quantitative stability of critical points of the fractional Sobolev inequality. We show that, for a non-negative function $u \in \dot H^s(\mathbb R^N)$ whose energy satisfies $$\tfrac{1}{2} S^\frac{N}{2s}_{N,s} \le \|u\|_{\dot H^s(\mathbb R^N)} \le \tfrac{3}{2}S_{N,s}^\frac{N}{2s},$$
where $S_{N,s}$ is the optimal Sobolev constant, the bound
$$    \|u -U[z,\lambda]\|_{\dot{H}^s(\mathbb R^N)} \lesssim \|(-\Delta)^s u - u^{2^*_s-1}\|_{\dot{H}^{-s}(\mathbb R^N)},
$$
holds for a suitable fractional Talenti bubble $U[z,\lambda]$. {For functions $u$ which are close to Talenti bubbles, we give the sharp asymptotic value of the implied constant in this inequality.} As an application {of this}, we derive an explicit polynomial extinction rate for positive solutions to a fractional fast diffusion equation.
\end{abstract}

\maketitle


\section{Introduction}
\label{sec:intro}

Let $N \in \N$ and $0 < s < N/2$. We consider {non-negative solutions to} the fractional critical elliptic problem 
\begin{align}\label{eq:elliptic}
(-\Delta)^su =  u^{p} \quad \text{on } \R^N,
\end{align}
where $p := 2^*_s -1 := \frac{2N}{N-2s} -1 = \frac{N+2s}{N-2s}$, and the fractional Laplace operator $\Ds u$ is defined through the Fourier representation 
\begin{equation}\label{frac lap def Fourier}
(-\Delta)^s u = \mathcal F^{-1} (|\xi|^{2s} \mathcal F u),
\end{equation} 
for any $u \in H^s(\R^N)$, i.e.  such that
\begin{align}\label{eq:fr_norm}
\Vert u \Vert_{L^2(\R^N)} + \left(\int_{\R^N} |\ds u|^2 \dd y \right)^{1/2} < \infty.
\end{align} 
We refer also to \cite{MR2944369,MR3967804,MR4480576} for further background about the fractional Laplacian and fractional Sobolev spaces. 

By \cite[Theorem 1.1]{Chen2006}, a positive function $u > 0$ satisfies  \eqref{eq:elliptic} if and only if $u$ equals the so-called \emph{Talenti bubble}
\begin{equation}\label{eq:talenti-s-bubble}
U [z,\lambda](x) = c_{N,s} \left( \frac{\lambda}{1 + \lambda^2|x-z|^2} \right)^\frac{N-2s}{2}, \qquad x \in \R^N, \end{equation}
for some $\lambda>0$, $z \in \R^N$, and 
$$c_{N,s} := 2^{2s}\left(\frac{\Gamma(\frac{N+2s}{2})}{\Gamma(\frac{N-2s}{2})}\right)^{\frac{N-2s}{4s}}.$$

Equation \eqref{eq:elliptic} corresponds to the Euler-Lagrange equation induced by the fractional Sobolev embedding ${ \dot H^s(\R^N) } \hookrightarrow L^{2^*_s}(\R^N)$. 
By \cite{Lieb1983} (see also \cite[Theorem 1.1]{MR2064421} and, for $s = 1$, the classical papers \cite{MR463908,MR448404}), the best constant in the fractional Sobolev embedding 
\begin{equation}
\label{eq:sob_ineq}
\int_{\R^N} |(-\Delta)^{s/2} u|^2 \dd y \geq S \| u\|_{L^\frac{2N}{N-2s}(\R^N)}^2 
\end{equation}
is given by  
\begin{equation}
\label{eq:sobolevconstant}
S_{N,s} := 2^{2s} \pi^s \frac{\Gamma(\frac{N+2s}{2})}{\Gamma(\frac{N-2s}{2})} \left( \frac{\Gamma(N/2)}{\Gamma(N)}\right)^{2s/N}
\end{equation}
and all optimizers are of the form $c U[z, \lambda]$, for some $c \in \R \setminus\{0\}$, $z \in \R^N$ and $\lambda > 0$. 

{Inequality \eqref{eq:sob_ineq} is valid for functions in the homogeneous Sobolev space $\dot{H}^s(\R^N)$, which is defined as the completion of $C^\infty_0(\R^N)$ with respect to 
\begin{align}\label{eq:h-fr_norm}
\|u\|_{\dot{H}^s(\R^N)} :=  \left(\int_{\R^N} |\ds u|^2 \dd  y \right)^{1/2} = \left(\int_{\R^N} |\widehat{u}(\xi)|^2 |\xi|^{2s} \dd  \xi \right)^{1/2}  .
\end{align} 
}

{If $s = 1$, then $\|u\|_{\dot{H}^1(\R^N)} = \|\nabla u\|_{L^2(\R^N)}$; in particular $\dot{H}^1(\R^N)$ is the completion of $C^\infty_0(\R^N)$ with respect to $\|\nabla u\|_{L^2(\R^N)}$.} 

{We recall that, by \cite[Proposition 3.6]{MR2944369}, if $s \in (0,1)$, we can also write the $\dot{H}^s(\R^N)$-norm as a double integral: namely,
\begin{equation}
    \label{doubleintegral}
    \iint_{\R^N\times \R^N} \frac{|u(x)-u(y)|^{2}}{|x-y|^{N+2 s}} \dd x \dd y = 2C(N,s)^{-1}\|(-\Delta)^{s/2}u\|_{\dot{H}^s(\R^N)}^2,
\end{equation}
with $$
C(N, s)=\left(\int_{\mathbb{R}^N} \frac{1-\cos \left(\zeta_1\right)}{|\zeta|^{N+2 s}} \dd \zeta\right)^{-1}.
$$
As pointed out in \cite{MR3586796}, we have, for some constant $c > 0$, 
$$\iint_{\mathbb{R}^N\times \R^N} \frac{|u(x)-u(y)|^{2}}{|x-y|^{N+2 s}} \dd x \dd y  =  \frac{c + o(1)}{1-s} \int_{\R^N} |\nabla u|^2 \dd x \quad  \text{ as $s \nearrow 1$.} $$
We denote by $\dot{H}^{-s}(\R^N)$ the dual space of $\dot{H}^s(\R^N)$ equipped with the norm 
\[ 
\|u\|_{\dot{H}^{-s}(\R^N)} := \left( \int_{\R^N} u (-\Delta)^{-s} u \dd y \right)^{1/2}=  \left(\int_{\R^N} |\widehat{u}(\xi)|^2 |\xi|^{-2s} \dd  \xi \right)^{1/2}. 
\]
}

\subsection{Stability estimates for critical points with explicit constants}

Once the optimal functions for \eqref{eq:sob_ineq}, resp. solutions to \eqref{eq:elliptic}, are characterized, it is reasonable to look for \emph{stability results}. Loosely speaking, this means that a function that is close to satisfying \eqref{eq:elliptic} (resp. optimizing \eqref{eq:sob_ineq}), in some sense to be made precise, must be close to the manifold
\begin{equation}
    \label{M definition}
    \mathcal M := \{ U[z, \lambda] \, : \, z \in \R^N, \, \lambda > 0 \} \subset \dot{H}^s(\R^N)
\end{equation}
of Talenti bubbles (resp. to the multiple of a Talenti bubble).

A breakthrough result concerning the stability of \eqref{eq:sob_ineq} was obtained by Bianchi and Egnell \cite{MR1124290}, who proved that the `Sobolev deficit' $\|\nabla u\|^2_{L^2(\R^N)} - S_{N,1} \|u\|^2_{L^{2^*}(\R^N)}$ is bounded from below in terms of some natural distance from the manifold of optimizers. More precisely, there exists a constant $c_{BE} > 0$ such that 
\begin{equation}
\label{bianchi egnell}
   R(u) = \frac{\|\nabla u\|^2_{L^2(\R^N)} - S_{N,1} \|u\|^2_{L^{2^*}(\R^N)}}{\dist(u,\widetilde{\mathcal M})^2} \ge c_{BE}  \quad \text{ for all } u \in \dot{H}^1(\R^N) ,
\end{equation}
 where $\widetilde{\mathcal M}$ is the manifold of optimizers of the Sobolev inequality (i.e. Talenti bubbles and their scalar multiples).  Here and in what follows, we denote, for $u \in \dot{H}^s(\R^N)$ and $\Omega \subset \dot{H}^s(\R^N)$, 
 \[\dist(u, \Omega) := \dist_{\dot{H}^s(\R^N)}(u, \Omega) := \inf \left\{ \left(\int_{\R^N} |(-\Delta)^{s/2} (u - v)|^2 \dd y\right)^\frac{1}{2} \, : \, v \in \Omega \right\}.
 \]
The Bianchi--Egnell inequality is the quantitative version  of a previous result by Lions in \cite{MR834360} (namely, that, if the Sobolev deficit is small for some function $f$, then $f$ has to be close to $\widetilde{\mathcal M}$) and answered a question posed by Brezis and Lieb in \cite[p. 75, Question B, point (c)]{MR790771}.

 Later, the result by Bianchi--Egnell was extended to biharmonic (see \cite{MR1694339}) and polyharmonic operators (see \cite{MR2018667}) as well; the fractional case was dealt with in \cite{MR3179693}. 
 
On the other hand, building on the work by Struwe \cite{MR760051}, in more recent contributions a similar quantitative stability analysis has been carried out for solutions to the equation \eqref{eq:elliptic} {for $s = 1$}. This can be seen as a generalization of the Bianchi--Egnell question in the sense that stability around an arbitrary critical point of the Sobolev quotient (instead of only minimizers) is considered. Here, the natural quantity replacing the Sobolev deficit in \eqref{bianchi egnell} is the $H^{-s}$ norm $\|(-\Delta)^s u - u^{2^*_s-1}\|_{\dot{H}^{-s}(\R^N)}$. For $s=1$, this was carried out in \cite{MR3873544} (in case of critical points with energy comparable to one Talenti bubble) and in \cite{MR4090466,2103.15360} (in the multi-bubble case). The fractional counterpart of the main theorems of \cite{MR4090466,2103.15360} has been recently given in \cite{2109.12610}.
 
An exciting open question, on which partial progress has been made recently, is to determine the value of the best constant $c_{BE}$ in \eqref{bianchi egnell}. The recent preprint \cite{2209.08651} contains the first explicit lower bound on $c_{BE}$, while in  \cite{2210.08482} the strict upper bound $c_{BE} < \frac{4}{N+4}$ is derived. We also mention \cite{ChLuTa} for an abstract result on best constants in stability inequalities. 

To the best of our knowledge, the corresponding question of explicit constants in the stability inequalities pertaining to \emph{critical points}, which we just described, has not been investigated so far. Filling this gap is one of the main purposes of the present paper. For critical points having the energy of roughly one Talenti bubble, we identify the explicit spectral constant corresponding to sequences close to $\mathcal M$. By a standard argument, this yields a `global' stability inequality with a non-explicit constant $c_{CP}^1(s)$.  Moreover, we can prove analogously to \cite{2210.08482} that $c_{CP}^1(s)$ must be strictly smaller than the spectral constant. We emphasize that while our results are valid for every fractional exponent $s \in (0, N/2)$, they are new even for $s =1$.

Let us finally mention that, as explained more thoroughly in the introduction of \cite{MR3873544}, almost-critical points $u$ of the Sobolev functional (in the sense that $\|(-\Delta)^s u - u^p\|_{\dot{H}^{-s}(\R^N)}$ is small) can in particular be viewed as metrics on $\mathbb S^N$ with almost constant fractional scalar curvature. For the notion of fractional scalar curvature -- and, generally, for the relevance of the fractional Laplacian in conformal geometry -- we refer to the survey \cite{MR3824214}.

\subsection{Explicit quantitative extinction rate for a fractional fast diffusion equation}
 
In a second step, we apply our findings on explicit stability constants to study the finite-time extinction for the fractional fast diffusion equation
\begin{align}   \label{eq:fde}
\begin{cases}
\partial_t u + (-\Delta)^s (|u|^{m-1}u) = 0, & t>0, \ x \in \R^N,  \\
u(0,x) = u_0(x), & x \in \R^N,
\end{cases}
\end{align}
with {$s \in (0,1)$ and}
\begin{align}\label{eq:m-fde}
    m:= \frac{1}{p} = \frac{1}{2^*_s-1} = \frac{N-2s}{N+2s}.
\end{align}
As shown in \cite[Theorems 2.3, 2.4, and 9.5]{MR2954615}, if {$u_0 \in L^1(\R^N) \cap L^p(\R^N)$ for some $p > q = (2_s^*)' :=  \frac{2N}{N+2s}$}, there exists a unique strong solution of \eqref{eq:fde}, which moreover vanishes in finite time (i.e., it has the \emph{finite-time extinction property}): more precisely, if $u_0$ is non-negative, there exists a $\bar T := T(u_0)$ such that $u(t,x) >0$ in $(0,\bar T)\times \R^N$ and $u(\bar T,\cdot) \equiv 0$. { Following \cite[Definitions 3.1 and 3.5]{MR2954615}, we call here $u$ a \emph{strong solution} of \eqref{eq:fde} if $u \in C([0, \infty); L^1(\R^N))$, $|u|^{m-1} u \in L^2_\text{loc}([0, \infty); \dot{H}^s(\R^N))$, $u(0,\cdot) = u_0$ a.e., $\partial_t u \in L^\infty((\tau, \infty); L^1(\R^N))$ for every $\tau > 0$, and  
\[ \int_{\R^N} \int_0^\infty u \partial_t \varphi   \dd t \dd x = \int_0^\infty \int_{\R^N} \ds |u|^{m-1} u \ds \varphi \dd x \dd t  \]
holds for all $\varphi \in C^1_0((0, \infty) \times \R^N)$. }

In case $s=1$, the extinction profile of \eqref{eq:fde} is studied in \cite{MR1857048}. For $s \in (0,1)$, analogous results are contained in \cite[Theorem 1.3]{MR3276166}: given $u_0 \in C^2(\R^N)$ such that $u_0 \ge 0$ and $|x|^{2s-N} u^m_0(x/|x|^2)$ {can be extended to a positive $C^2$-function in the origin}, there exist $z \in \R^N$ and $\lambda >0$ such that 
\begin{align}\label{eq:asy}
\left\|\frac{u(t,\cdot)}{U_{\bar T, z, \lambda}(t,\cdot)} - 1 \right\|_{L^\infty(\R^N)} \to 0 \qquad \text{ as } t \to \bar T^-,
\end{align}
where 
\begin{align}\label{eq:extinction-prof}
    U_{\bar T,z,\lambda}(t,x):= \left(\frac{p-1}{p}\right)^{\frac{p}{p-1}}(\bar T -t)^{\frac{p}{p-1}} U[z,\lambda]^{p}, \qquad t \in (0,\bar T), \ x \in \R^N,
\end{align}
with $\frac{p}{p-1} = \frac{1}{1-m} = \frac{N+2s}{4s}$.
As an application of our single-bubble stability result, we shall establish a convergence rate of the type \begin{align*}
\left\|\frac{u(t,\cdot)}{U_{\bar T, z, \lambda}(t,\cdot)} - 1 \right\|_{L^\infty(\R^N)} \le C_* (\bar T-t)^{\kappa_{N,s}} \qquad \text{ for all } t \in (0,\bar T),
\end{align*}
where $\kappa_{N,s}$ is a universal constant and $C_* > 0$ is a constant depending additionally on the initial datum $u_0$. This result is proved in \cite[Theorem 5.1]{MR4090466} (see also \cite[Theorem 1.3]{MR3873544}) in case $s=1$.  The attention to the explicit stability constant in the result for the fractional Sobolev inequality yields, in turn, an explicit estimate for the convergence rate in the fractional fast diffusion asymptotics. 

For fast-diffusion exponents in the range $0<m \in (\frac{N-2s}{N+2s},1)$, the problem of studying the convergence rates to the corresponding extinction profile  has also a long history (in case $s=1$): we refer to \cite{MR3307161,MR2481073, MR1986060, MR1940370, MR2246356} for the Cauchy problem in $\R^N$ and to \cite{MR4221933, Akagi2021, 2008.01311, Choi2022} for the equation posed in a bounded domain $\Omega \subset \R^N$.

\subsection{Outline of the paper}
\label{ssec:outline}

In Section \ref{sec:main}, we state our main results: the quantitative stability of the fractional Talenti bubble and the application to the asymptotics for the fractional fast-diffusion equation. Finally, in Sections \ref{sec:proof-stability1} and \ref{sec:yamabe}, we develop the proofs.

\subsection*{Notation}
For a set $M$ and functions $f,g : M \to \R_+$, we shall write $f(m) \lesssim g(m)$ if there exists a constant $C > 0$, independent of $m$, such that $f(m) \leq C g(m)$ for all $m \in M$, and accordingly for $\gtrsim$.

\section{Main results}
\label{sec:main}

\subsection{Quantitative stability of the fractional Sobolev inequality near one Talenti bubble}
\label{ssec:talenti-1}

 {The starting point of our analysis is the following consequence of \cite[Théorème 1.1]{Gerard1998}, which can be seen as a fractional generalization of Struwe's celebrated result \cite[Proposition 2.1]{MR760051} for $s = 1$. A very similar statement for a smooth and bounded set $\Omega \subset \R^N$ appears in \cite[Theorem 1.1]{MR3316602}.}
 

\begin{lemma}[Qualitative stability for one bubble]\label{lm:s-struwe}
Let $N \in \N$, $0 < s < N/2$, and $\{u_k\}_{k \in \mathbb N}$ such that 
\begin{align}
   \label{eq:bound-energy} & \frac{1}{2} S^\frac{N}{2s}_{N,s} \le \int_{\R^N}|(-\Delta)^{s/2} u_k|^{2} \dd x \le \frac{3}{2}S^\frac{N}{2s}_{N,s},\\ 
   &  {u_k \geq 0},  \label{eq:u-nonneg} \\
  \label{eq:ps-conv}  &\|(-\Delta)^s u_k - u_k^{2^*_s-1}\|_{\dot{H}^{-s}(\R^N)} \to 0 \quad \text{ as } k \to +\infty.
\end{align}
Then there exist sequences $\{z_k\}_{k \in \N} \subset \mathbb R^N$ and $\{\lambda_k\}_{k \in \N} \subset (0,\infty)$ such that
$$\|u_k-U[z_k,\lambda_k]\|_{\dot{H}^s(\R^N)} \to 0  \quad \text { as } k \to +\infty.$$ 
\end{lemma}

{We defer the proof of Lemma \ref{lm:s-struwe} to Section \ref{sec:proof-stability1} below.} 

\begin{remark}[Palais-Smale condition]
The conditions \eqref{eq:bound-energy} and \eqref{eq:ps-conv} express the fact that $\{u_k\}_{k \in \N}$ is a Palais-Smale sequence for the functional 
$$J(u) = \frac{1}{2}\int_{\R^N} |(-\Delta)^{s/2}u|^2 \dd x - \frac{1}{2^*_s} \int_\Omega |u|^{2^*_s} \dd x.$$
\end{remark}

Our first main result is the following quantitative stability estimate for single fractional Talenti bubbles.

We start by stating a version of the stability estimate that only applies to sequences satisfying the full set of hypotheses of Lemma \ref{lm:s-struwe}, in particular the Palais-Smale type condition \eqref{eq:ps-conv}. In exchange, for these sequences, we can estimate the stability constant explicitly. Besides being of interest of itself, this allows us to get an explicit estimate for the convergence rate in our second main result, Theorem \ref{th:main-fde} below.  

To state our result, we introduce the space
\begin{equation}
    \label{definition Txl}
     T_\xl = \text{span} \left \{ U[x,\lambda], \pl U[x,\lambda], \{\pxi U[x,\lambda]\}_{i=1}^N \right\} \subset \dot{H}^s(\R^N) 
\end{equation}
and denote by $T_\xl^\perp \subset \dot{H}^s(\R^N)$ its orthogonal complement in $\dot{H}^s(\R^N)$ with respect to the scalar product $\langle u,v \rangle_{\dot{H}^s(R^N)} := \int_{\R^N} (-\Delta)^{s/2} u (-\Delta)^{s/2} v \dd y$. 
Notice that $U[x, \lambda]$, $\partial_\lambda U[x, \lambda]$, and $\partial_{x_i} U[x,\lambda]$ are pairwise orthogonal with respect to $\langle \emptyparam, \emptyparam \rangle_{\dot{H}^s(R^N)}$. 

\begin{theorem}[Quantitative stability for Palais-Smale sequences]
\label{th:main-stability-explicit-constant}
Let $N \in \N$, $0 < s < N/2$, and $\{u_k\}_{k \in \N}$ satisfy \eqref{eq:bound-energy}-\eqref{eq:ps-conv}. Then 
\begin{equation}
    \label{asymptotic inequality thm}
    \liminf_{k \to \infty} \frac{ \|(-\Delta)^s u_k - u_k^{2^*_s-1}\|_{\dot{H}^{-s}(\R^N)}  }{\dist(u_k, \mathcal M) } \geq \frac{4s}{N+2s+2} =: \gamma_{N,s} . 
\end{equation} 
Moreover, equality holds if and only if the following statement holds: given $z_k$ and $\lambda_k$ such that $\|u_k - U[z_k, \lambda_k]\|_{\dot{H}^s(\R^N)} = \dist(u_k, \mathcal M)$ (see Lemma \ref{lemma optimal z lambda}), then the unique functions $v_k \in T_{z_k, \lambda_k}$ and $\rho_k \in T_{z_k, \lambda_k}^\bot$ such that $u_k - U[z_k, \lambda_k] = v_k + \rho_k$ satisfy (up to taking a subsequence) 
\begin{align}
    \label{optimality condition 1 thm}
    &\|v_k\|_{\dot{H}^{s}(\R^N)} = o(\| \rho_k\|_{\dot{H}^{s}(\R^N)}), \\
    &\frac{\norm{\rho_k}_{\dot{H}^{s}(\R^N)}^2 - p \int_{\R^N} U[z_k,\lambda_k]^{p-1} \rho_k^2 \dd y}{\norm{\rho_k}_{\dot{H}^{s}(\R^N)}^2}  \to  \frac{4s}{N+2s+2}  . \label{optimality condition 2 thm}
\end{align} 
\end{theorem}

The corresponding asymptotic inequality in the Bianchi--Egnell setting, i.e. 
\begin{equation}
    \label{asymptotic inequality bianchi egnell}
    \liminf_{k \to \infty} \frac{\| \ds u\|^2_{\dot{H}^s(\R^N)} - S_{N,s} \|u\|^2_{2^*_s}}{\text{dist}(u_k, \widetilde{\mathcal M})^2} \geq \gamma_{N,s}, 
\end{equation}
for all sequences $\{u_k\}_{k \in \N}$ such that $\text{dist}(u_k, \widetilde{\mathcal M}) \to 0$ as $k \to \infty$, is well-known; see, e.g., \cite[Proposition 2]{MR3179693} for an equivalent statement on $\mathbb S^N$. It can  be proved by means of the coercivity inequality from Lemma \ref{lm:coercivity}. While this inequality is still at the core of the proof of Theorem \ref{th:main-stability-explicit-constant}, the proof as a whole becomes somewhat more involved due to the presence of the $\dot{H}^{-s}$ norm in \eqref{asymptotic inequality thm}. 

Note that the condition \eqref{optimality condition 2 thm} means precisely that $\rho_k$ optimizes asymptotically the coercivity inequality \eqref{coerc_ineq_rey}. 

It is curious to remark that the constant $\gamma_{N,s}$ in inequalities \eqref{asymptotic inequality thm} and \eqref{asymptotic inequality bianchi egnell} are the same, even though the left side of \eqref{asymptotic inequality bianchi egnell} is quadratic, but that of \eqref{asymptotic inequality thm} is not. What is morally at the origin of this is the fact   that, in \eqref{asymptotic inequality bianchi egnell}, the quotient is between an $\dot{H}^{-s}$ norm and a $\dot{H}^s$ norm, whereas in \eqref{asymptotic inequality bianchi egnell} it is between two $\dot{H}^s$ norms. For more details, see the proof of Theorem \ref{th:main-stability-explicit-constant}. 

From Theorem \ref{th:main-stability-explicit-constant} we can deduce, by a compactness argument, the following general stability result for all functions satisfying the bound \eqref{eq:bound-energy}. 

\begin{theorem}[Quantitative stability: single bubble case]\label{th:main-stability1}
Let $N \in \N$ and $0 < 2s < N$.  
Then, 
\begin{align}
\label{error bound thm}
  c^1_{CP}(s):= \inf_{u \in \mathcal B} \frac{ \|(-\Delta)^s u - u^{2^*_s-1}\|^2_{\dot{H}^{-s}(\R^N)}  }{\dist(u, \mathcal M)^2 } > 0, 
\end{align}
where $S_{N,s}$ is defined in \eqref{eq:sobolevconstant} and 
\begin{align*}
    \mathcal B:= \left\{u \in \dot H^s(\R^N): u \ge 0, \  \frac{1}{2} S^\frac{N}{2s}_{N,s} \le \|u\|_{\dot H^s(\mathbb R^N)}^2 \le \frac{3}{2}S_{N,s}^\frac{N}{2s}\right\}.
\end{align*}
Furthermore, if $N\ge 2$, then $c^1_{CP}(s) < \gamma_{N,s}^2$.
\end{theorem}

{
\begin{remark}[Non-negativity assumption]
    We point out that, when $s \in (0,1]$, then Lemma \ref{lm:s-struwe} and Theorems \ref{th:main-stability-explicit-constant} and \ref{th:main-stability1} (with $u^p$ replaced by $|u|^{p-1} u$) also hold without the assumption that $u_k$, resp. $u$, is non-negative (indeed, the proofs  go through unchanged). On the other hand, in the proof of Lemma \ref{lm:s-struwe} the energy bound \eqref{eq:bound-energy} keeps $u_k$ from approaching a sign-changing solution. This is explained in more detail in Remark \ref{remark sign-changing} below.
\end{remark} }

As already mentioned above, the analogue of the strict inequality $c^1_{CP}(s) < \gamma_{N,s}^2$ for the fractional Bianchi--Egnell constant  in \eqref{asymptotic inequality bianchi egnell} was recently proved in \cite{2210.08482}. The key idea of the proof -- namely, a \emph{third order} expansion of the quotient together with a clever choice of a test function -- works here as well. {It is noteworthy and somewhat curious that this choice of test function (see \eqref{test function v}) requires the additional assumption $N \geq 2$. Whether $c^1_{CP}(s) < \gamma_{N,s}^2$ when $N = 1$ and $s \in (0, 1/2)$ remains unclear to us, analogously to the situation studied in \cite{2210.08482}.} 

The study of the behavior of the best constant $c^k_{CP}(s)$ for the analogous problem with higher energies $\frac{2k-1}{2} S^\frac{N}{2s}_{N,s} \le \|u\|_{H^s(\mathbb R^N)}^2 \le \frac{2k+1}{2}S_{N,s}^\frac{N}{2s}$, for any $k \geq 2$, seems to be an interesting open question. {This setting corresponds to functions $u \geq 0$ which look approximately like $k$ weakly interacting Talenti bubbles. Since $k$ such bubbles form a very good approximation to a solution of \eqref{eq:elliptic}, it is again of interest to control the ratio $\inf_{u \in \mathcal B} \|(-\Delta)^s u - u^{2^*_s-1}\|^2_{\dot{H}^{-s}(\R^N)}  \dist(u, \mathcal M_k)^{-\alpha}$, for a suitable power $\alpha$ and $\mathcal M_k$ denoting the manifold of sums of $k$ Talenti bubbles. Surprisingly, it has been found in \cite{MR4090466, 2103.15360, 2109.12610} that the suitable exponent $\alpha$ is different from $2$ in general. What seems to remain completely unclear for now is the analogue of Theorem \ref{th:main-stability-explicit-constant}, i.e. to find the best asymptotic constant (replacing $\gamma_{N,s}$) is for sequences $\{u_n\}_{n \in \N}$ converging to $\mathcal M_k$, and to characterize the $\{u_n\}_{n \in \N}$ achieving this constant.} We hope to come back to this question in future work.

\subsection{Application to the study of fractional fast-diffusion equations}
\label{ssec:fast}

As an application of Theorem \ref{th:main-stability1}, we study the quantitative convergence to equilibrium of the fast diffusion equation \eqref{eq:fde}. In contrast to the previous section, here we need to specify that the fractional order satisfies $s\in (0, \min\{1, N/2\}$).

Our starting point is the following result on the extinction profiles of \eqref{eq:fde} (contained in \cite[Theorem 1.3]{MR3276166}, which is the fractional counterpart of \cite[Theorem 1.1]{MR1857048}). Conceptually, it plays the same role as Lemma \ref{lm:s-struwe} in the proof of Theorem \ref{th:main-stability1}.

\begin{lemma}[Extinction profile of the fractional fast diffusion equation]\label{lm:s-pino}
Let $N \in \N$ and $s \in (0,\min\{1,N/2\})$. Let ${0 \not \equiv u_0} \in C^2(\R^N)$ such that $u_0 \ge 0$ and $|x|^{2s-N} u^m_0({x}/{|x|^2})$ can be extended to a positive $C^2$ function in the origin. Then, there exist an extinction time $\bar T = T(u_0) \in (0,\infty)$ such that the solution $u$ of \eqref{eq:fde} satisfies $u(t,x) >0$ for $t \in (0,\bar T)$ and $u(\bar T, \cdot) \equiv 0$. Moreover, there exist $z \in \R^N$ and $\lambda >0$ such that 
\begin{align}\label{eq:asy}
\left\|\frac{u(t,\cdot)}{U_{\bar T, z, \lambda}(t,\cdot)} - 1 \right\|_{L^\infty(\R^N)} \to 0 \quad \text{ as } t \to \bar T^-,
\end{align}
where 
\begin{align}\label{eq:extinction-prof}
    U_{\bar T,z,\lambda}(t,x):= \left(\frac{p-1}{p}\right)^{\frac{p}{p-1}}(\bar T -t)^{\frac{p}{p-1}} U[z,\lambda]^{p}(x), \qquad t \in (0, \bar T), \ x \in \R^N,
\end{align}
with $\frac{p}{p-1} = \frac{1}{1-m} = \frac{N+2s}{4s}$.
\end{lemma}

\begin{remark}[Explicit solutions in separated variables with extinction in finite time]
As observed in \cite[Remark, p. 1275]{MR2954615}, $U_{\bar T, z, \lambda}(t,\cdot)$ is an explicit
family of solutions of \eqref{eq:fde}-\eqref{eq:m-fde} in separated variables with extinction in finite time.
\end{remark}

Next, we state our result on the rates of convergence to equilibrium for fractional fast diffusion equations. 

\begin{theorem}[Convergence to equilibrium for fractional fast diffusion equations]\label{th:main-fde}
Let $N \in \N$ and $s \in (0,\min\{1,N/2\})$. Let ${0 \not \equiv u_0} \in C^2(\R^N)$ such that $u_0 \ge 0$ and $|x|^{2s-N} u^m_0(x/|x|^2)$ can be extended to a positive $C^2$ function in the origin. Then, there exist an extinction time $\bar T = T(u_0) \in (0,\infty)$ such that the solution $u$ of \eqref{eq:fde} satisfies $u(t,x) >0$ for $t \in (0,\bar T)$ and $u(\bar T, \cdot) \equiv 0$. Moreover, there exist $z \in \R^N$ and $\lambda >0$, such that 
\begin{align}\label{eq:asy-q}
\left\|\frac{u(t,\cdot)}{U_{\bar T, z, \lambda}(t,\cdot)} - 1 \right\|_{L^\infty(\R^N)} \le C_* (\bar T-t)^{\kappa} \qquad \forall \, t \in (0,\bar T),
\end{align}
for some $C_* > 0$ depending on the initial datum $u_0$ and all $\kappa < \kappa_{N,s}$, where
\begin{equation}
    \label{kappa definition}
    \kappa_{N,s} := \frac{1}{(N+2-2s)(p-1)} \gamma_{N,s}^2,
\end{equation}
with $p = \frac{N+2s}{N-2s}$ and $\gamma_{N,s} = \frac{4s}{N+2s+2}$ as in Theorem \ref{th:main-stability-explicit-constant}. 
\end{theorem}

The value \eqref{kappa definition} is, to our knowledge, the first explicit lower bound on the optimal extinction rate for the fractional fast diffusion equation \eqref{eq:fde}. This is new even for the non-fractional case $s =1$ treated in \cite{MR3873544}, where \eqref{eq:asy} is only given for some non-explicit constant $\kappa > 0$. The main point which permits us to get an explicit constant is to derive and use Theorem  \ref{th:main-stability-explicit-constant} with the sharp constant, instead of only Theorem \ref{th:main-stability1}. 

For the case $s = 1$, similar explicit extinction rates for the fast diffusion equation with exponent $m \in (\frac{N-2s}{N+2s},1)$ on bounded domains have been obtained recently in \cite{MR4221933, Akagi2021, Choi2022}. In these works, the derived extinction rate turns out to be optimal, in the sense that it coincides with the spectral gap of a suitable operator associated with the linearization of the fast diffusion equation. On the other hand, it turns out that the constant $\kappa_{N,s}$ from Theorem \ref{th:main-fde} does not have the same sharpness property even if we have plugged in the sharp constant from Theorem \ref{th:main-stability-explicit-constant}. We refer to Remark \ref{remark sharpness} for a more detailed discussion of this issue.

\section{Proof of the stability result for one fractional Talenti bubble}
\label{sec:proof-stability1}

{We begin by giving the proof of the qualitative convergence stated in Lemma \ref{lm:s-struwe}. }

\begin{proof}
    [Proof of Lemma \ref{lm:s-struwe}]
  {  By \eqref{eq:bound-energy}, $\{u_k\}_{k \in \N}$ is a bounded sequence in $\dot{H}^s(\R^N)$. By the profile decomposition given in \cite[Théorème 1.1]{Gerard1998}, for every $j \in \N$ there are functions $\psi_j \in \dot{H}^s(\R^N)$ and sequences $\{h_k^{(j)}\}_{k \in \N} \subset (0, \infty)$ and $\{x_k^{(j)}\}_{k \in \N} \subset \R^N$ satisfying 
    \begin{equation}
        \label{weakly interacting}
        \left| \log \frac{h_k^{(i)}}{h_k^{(j)}} \right| + \frac{|x_k^{(i)} - x_k^{(j)}|}{h_k^{(i)}} \to \infty \qquad \text{ as } k \to \infty, \qquad \text{ for all } i \neq j,
    \end{equation} 
    such that, up to passing to a subsequence, for any $\ell \in \N$, $u_k$ can be written as 
    \begin{equation}
        \label{profile dec}
        u_k= \sum_{j = 1}^\ell h_k^{-\frac{N}{2_s^\ast}} \psi_j\left(\frac{ \cdot - x_k^{(j)}}{h_k^{(j)}}\right) + r_k^{(\ell)}, 
    \end{equation}
    where the remainder term $r_k^{(\ell)}$ satisfies $\limsup_{k \to \infty} \|r_k^{(\ell)}\|_{L^{2_s^*}(\R^N)} \to 0$ as $\ell \to \infty$. Moreover, the $\dot{H}^s$-norm decomposes as 
    \begin{equation}
        \label{Hs norm profiles}
        \|u_k\|_{\dot{H}^s(R^N)}^2 = \sum_{j = 1}^\ell \|\psi_j\|_{\dot{H}^s(\R^N)}^2 + \|r_k^{(\ell)}\|_{\dot{H}^s(\R^N)}^2 + o(1) \qquad \text{ as } \ell \to \infty.  
    \end{equation} 
    Since $\|u_k\|_{\dot{H}^s(\R^N)}$ is bounded from below by \eqref{eq:bound-energy}, by \eqref{Hs norm profiles} not all the $\psi_j$ are zero. So we may assume, up to relabeling, that $\psi_1 \not \equiv 0$.}

{As a consequence of the Palais-Smale condition \eqref{eq:ps-conv}, every $\psi_j$ is a solution to 
\begin{equation}
    \label{sobolev lemma proof}
    (-\Delta)^s \psi = |\psi|^{2_s^\ast - 2} \psi \qquad \text{ on } \R^N,
\end{equation} 
in the weak sense. Indeed, fix a test function $\phi$ and let  $(\cdot, \cdot)$ denote the dual bracket between $\dot{H}^s$ and $\dot{H}^{-s}$. Then for $\phi_k^{(j)} := h_k^{-\frac{N}{2_s^\ast}} \phi \left(\frac{ \cdot - x_k^{(j)}}{h_k^{(j)}}\right)$, as a consequence of \eqref{weakly interacting}, we get 
\begin{align*}
    o(1) &= ((-\Delta)^s u_k - u_k^{2^*_s-1}, \psi_k^{(j)}) \\
    &= ((-\Delta)^s \left( h_k^{-\frac{N}{2_s^\ast}} \psi_j\left(\frac{ \cdot - x_k^{(j)}}{h_k^{(j)}}\right) \right) -  \left( h_k^{-\frac{N}{2_s^\ast}} \psi_j\left(\frac{ \cdot - x_k^{(j)}}{h_k^{(j)}}\right) \right) ^{2^*_s-1}, \phi_k^{(j)}) + o(1) \\
    &= ((-\Delta)^s \psi_j - \psi_j^{2^*_s-1}, \phi) + o(1)
\end{align*}
and thus $((-\Delta)^s \psi_j - \psi_j^{2^*_s-1}, \phi) = 0$ for every $\phi \in \dot{H}^s(\R^N)$, which is \eqref{sobolev lemma proof} as claimed. }

{Now every weak solution $\psi$ to \eqref{sobolev lemma proof} satisfies $\| \psi \|_{\dot{H}^s(\R^N)}^2 = \|\psi\|_{2_s^*}^{L^{2_s^*}(\R^N)}$, so that 
\[ S_{N,s} \leq \frac{\| \psi \|_{\dot{H}^s(\R^N)}^2}{\|\psi\|_{2_s^*}^{2}} = \| \psi \|_{\dot{H}^s(\R^N)}^{2 - \frac{4}{2_s^*}} =  \| \psi \|_{\dot{H}^s(\R^N)}^{\frac{4s}{N}},  \]
that is, $\| \psi \|^2_{\dot{H}^s(\R^N)} \geq S_{N,s}^\frac{N}{2s}$. In view of \eqref{Hs norm profiles}, in order to fulfill the energy bound \eqref{eq:bound-energy}, we must have $\psi_j \equiv 0$ for all $j \geq 2$. In turn, $r_k^{(\ell)} = r_k^{(1)} =: r_k$ is independent of $\ell$ and we find $\limsup_{k \to \infty} \|r_k^{(1)}\|_{L^{2_s^*}(\R^N)} = 0$. Denoting $\eta_k :=  h_k^{-\frac{N}{2_s^\ast}} \psi_1\left(\frac{ \cdot - x_k^{(1)}}{h_k^{(1)}}\right)$, we thus know that $\|u_k - \eta_k\|_{L^{2_s^*}(\R^N)} \to 0$ as $k \to \infty$. Since $u \geq 0$ by \eqref{eq:u-nonneg}, this implies $\psi_1 \geq 0$. The classification of non-negative solutions to \eqref{sobolev lemma proof} from \cite{Chen2006} then gives $\psi_1 = U[z_0, \lambda_0]$ for some $z_0 \in \R^N$, $\lambda_0 > 0$. }

{Finally, we upgrade this convergence to $\|u_k - \eta_k\|_{\dot{H}^s(\R^N)} \to 0$ using \eqref{eq:ps-conv}. To do so, we write 
\begin{align*}
    \|u_k - \eta_k\|_{\dot{H}^s(\R^N)}^2 &= ( (-\Delta)^s u_k, u_k - \eta_k) - ((-\Delta)^s \eta_k, u_k - \eta_k),
\end{align*}
where $(\cdot, \cdot)$ again denotes the dual bracket between $\dot{H}^s$ and $\dot{H}^{-s}$. 
Now, we can write 
\begin{align*}
    ( (-\Delta)^s u_k, u_k - \eta_k) &= (u_k^{2_s^*-1}, u_k - \eta_k) + \mathcal O(\|(-\Delta)^s u_k - u_k^{2_s^*-1}\|_{\dot{H}^{-s}(\R^N)} \|u_k - \eta_k\|_{\dot{H}^s(\R^N)}) \\
    &= (u_k^{2_s^*-1}, u_k - \eta_k) + o(1) 
\end{align*} 
by \eqref{eq:ps-conv} and the fact that $u_k$ and $\eta_k$ are bounded in $\dot{H}^s$. }

{Since $\psi_1$ solves \eqref{sobolev lemma proof}, so does its rescaling $\eta_k$, and we altogether get 
\[  \|u_k - \eta_k\|_{\dot{H}^s(\R^N)}^2 = (u_k^{2^*_s-1}, u_k - \eta_k) + (|\eta_k|^{2^*_s -2} \eta_k, u_k -\psi-k) + o(1) = o(1) \]
as desired, as a consequence of the convergence $\|u_k - \eta_k\|_{L^{2_s^*}(\R^N)} \to 0$ (and boundedness of $u_k$ and $\eta_k$ in $L^{2^*_s}$ by Sobolev embedding). }

{If the statement of Lemma \ref{lm:s-struwe} were not true for the entire sequence $\{u_k\}_{k \in \N}$, then it would fail along a subsequence $\{u_{k_n}\}_{k_n \in \N}$, i.e. $\inf_{z, \lambda} \| u_{k_n} - U[z, \lambda] \|_{\dot{H}^s(\R^N)} \geq \eps_0$ for all $n$, for some $\eps_0 > 0$.  Applying the previous reasoning to this subsequence would then lead to a contradiction. Hence, Lemma \ref{lm:s-struwe} actually holds for the entire sequence $u_k$, as claimed.  }
\end{proof}

\begin{remark}[The energy of sign-changing solutions and the role of  the assumption $u_k \geq 0$]
\label{remark sign-changing}
{We now explain why, for $s \in (0,1]$, we are actually able to drop the non-negativity assumption \eqref{eq:u-nonneg} in Lemma \ref{lm:s-struwe}. We choose to only present this as an aside in order to keep the statement of Lemma \ref{lm:s-struwe} more readable, and because in our main application, Theorem \ref{th:main-fde}, we deal with non-negative functions only anyway. }

{
    Let $s \in (0, 1]$ and  $\psi \in H^s(\R^N)$ be a sign-changing  solution of \eqref{sobolev lemma proof}. Then its positive and negative parts, $\psi_\pm$, are both non-trivial and in $\dot{H}^s(\R^N)$ themselves. For $s \in (0,1)$, we denote by 
    \[ I_s(u,v) := \frac{1}{2}  C(N,s) \iint_{\R^N \times \R^N} \frac{(u(x) - u(y))(v(x) - v(y))}{|x-y|^{N+2s}} \dd x \dd y \]  
 the double-integral from \eqref{doubleintegral}, so that $\|u\|_{\dot{H}^s(\R^N)}^2 = I(u,u)$. Testing \eqref{sobolev lemma proof} with $\psi_+$ and $-\psi_-$ and using the fact that  $\psi = \psi_+ -\psi_-$ yields 
 \[ \|\psi_\pm\|_{\dot{H}^s(\R^N)}^2 - I_s(\psi_+, \psi_-) = \|\psi_\pm\|_{2_s^*}^{2_s^*}.  \]
 Moreover, we can directly check that $I_s(\psi_+, \psi_-) < 0$ because $\psi_+$ and $\psi_-$ are non-negative and have disjoint support. As a consequence, similarly to the proof of Lemma \ref{lm:s-struwe}, we can estimate
 \[ S_{N,s} \leq \frac{\|\psi_\pm\|_{\dot{H}^s(\R^N)}^2   }{\|\psi_\pm\|_{L^{2_s^*}(\R^N)}^{2_s^*}} < \frac{\|\psi_\pm\|_{\dot{H}^s(\R^N)}^2 - I_s(\psi_+, \psi_-)}{\|\psi_\pm\|_{L^{2_s^*}(\R^N)}^{2_s^*}}  = (\|\psi_\pm\|_{\dot{H}^s(\R^N)}^2 - I_s(\psi_+, \psi_-))^{1 - \frac{2}{2_s^*}}.  \]
Hence, $\|\psi_\pm\|_{\dot{H}^s(\R^N)}^2 - I_s(\psi_+, \psi_-) > S_{N,s}^\frac{N}{2s}$ and we get 
\begin{equation}
    \label{signchanging bound}
    \|\psi \|_{\dot{H}^s(\R^N)}^2  = \|\psi_+ \|_{\dot{H}^s(\R^N)}^2  +  \|\psi_- \|_{\dot{H}^s(\R^N)}^2 - 2 I_s(\psi_+, \psi_-) \geq 2 S_{N,s}^\frac{N}{2s}.  
\end{equation} 
Since the function $\psi_1$ in the proof of Lemma \ref{lm:s-struwe} satisfies $\| \psi_1\|_{\dot{H}^s(\R^N)}^2 \leq \frac{3}{2} S_{N,s}^\frac{N}{2s}$ as a consequence of \eqref{eq:bound-energy} and \eqref{Hs norm profiles}, we see from \eqref{signchanging bound} that $\psi_1$ cannot be sign-changing. As in the above proof we can then deduce $\psi_1 = \pm U[z_0, \lambda_0]$ for some $z_0 \in \R^N$, $\lambda_0 > 0$. So $\psi_1$ turns out to have a sign automatically,  even without having made the assumption that $u_k$ has a sign. For $s = 1$, the argument is similar but simpler; we omit the details.  This concludes our remark. }
\end{remark}

We now pursue the proof of our main results. Recall that the space $T_{z,\lambda}$ has been defined in \eqref{definition Txl}. We now state a fundamental inequality valid for functions in its $\dot{H}^s$-orthogonal complement $T_{z, \lambda}^\perp$. For $s=1$, this inequality is due to Rey in \cite{Rey1990}. For every $s \in (0, N/2)$, its equivalent version on the sphere $\mathbb S^N$ appears in \cite{MR3179693}.  

\begin{lemma}[Coercivity inequality]\label{lm:coercivity}
For all $z \in \R^N$ and $\lambda > 0$, we have 
\begin{equation}
\label{coerc_ineq_rey}
\norm{v}_{\dot{H}^{s}(\R^N)}^2 - p \int_{\R^N} U[z,\lambda]^{p-1} v^2 \dd y \geq \frac{4s}{N+2s+2}  \norm{v}_{\dot{H}^{s}(\R^N)}^2,
\end{equation}
for all $v \in T_{z,\lambda}^\bot$. {Moreover, the constant $\frac{4s}{N+2s+2}$ is optimal.}
\end{lemma}

In particular, Lemma \ref{lm:coercivity} implies that 
\begin{equation}
\label{[v,w] definition}
[v, w] := \norm{v}_{\dot{H}^{s}(\R^N)}^2 - p \int_{\R^N} U[z,\lambda]^{p-1} v^2 \dd y
\end{equation}
defines a scalar product on $T_{z, \lambda}^\bot$. 

\begin{proof}
{Firstly, by conformal invariance it is enough to prove the inequality for $z = 0$ and $\lambda = 1$ only. By means of the (inverse) stereographic projection $\mathcal S: \R^N \to \mathbb S^N$, we can reduce the inequality to the spectral gap on $\S^N\subseteq\R^{N+1}$. Indeed, it can be justified using the Funk-Hecke formula (see, e.g., \cite[Eq. (25) and Lemma 6]{Frank2023}) that, for $V: \mathbb S^N \to \R$ and $v$ related by
\begin{equation}
    \label{ster proj}
    v(x) = V(\mathcal S(x)) (\det D\mathcal S(x))^\frac{1}{2_s^*} ,
\end{equation}
we have 
\[ \norm{v}_{\dot{H}^{s}(\R^N)}^2 = \scalprod{V, A_s V}_{L^2(\mathbb S^N)}   \] 
for a certain operator $A_s$ acting on $L^2(\mathbb S^N)$. This operator is diagonal with respect to spherical harmonics  (i.e., eigenfunctions of $(-\Delta)_{\mathbb S^N}$) and acts on a spherical harmonic $Y_k$ of degree $k$ by $A_s Y_k = \alpha(k) Y_k$, where the eigenvalue $\alpha(k)$ is given by 
\[ \alpha(k) = \frac{\Gamma(k + \frac{n}{2} + s)}{\Gamma(k + \frac{n}{2} - s)}.  \]
Moreover, the second term in \eqref{coerc_ineq_rey} simply transforms to 
\[ p \int_{\R^N} U[0,1]^{p-1} v^2 \dd y = \alpha(1) \int_{\mathbb S^N} V^2 \dd S. \]
Finally, a computation shows that the condition $v\in T_{0,1}^\bot$ is equivalent to $V:\S^N\to\R$ being $L^2(\mathbb S^N)$-orthogonal to the spherical harmonics $1, \omega_1, ..., \omega_{N+1}$ of degree $0$ and $1$. Inequality \eqref{coerc_ineq_rey} is now equivalent, via these transformations, to the simple spectral inequality 
\begin{equation}
    \label{spectral ineq}
    \scalprod{V, A_s V} \geq \alpha(2) \|V\|_{L^2(\mathbb S^N)}^2,
\end{equation} which holds because $\alpha(k)$ is an increasing sequence. The claimed sharpness of the constant now follows from the facts that \eqref{spectral ineq} is an equality for $V = Y_2$ and that $1 - \frac{\alpha(1)}{\alpha(2)} = \frac{4s}{N+2s+2}$.  }

We refer to \cite[Proof of Proposition 3.4]{DNK21a} for further details.
\end{proof}

\begin{remark}[The role of the linearized operator and its eigenfunctions]\label{rk:linearized}
We shall write $U:= U[0,1]$ for short in what follows. { The discussion from the proof of Lemma \ref{lm:coercivity} can be equivalently rephrased on $\R^N$ by saying that the operator 
\[ \mathcal L := \frac{(-\Delta)^s }{p U^{p-1}} \]
acting on the weighted $L^2$-space $\mathcal H := L^2(p U^{p-1} \dd x)$ has a countable number of eigenvalues $\{\mu_i\}_{i=1}^\infty$ and eigenfunctions for different eigenvalues are orthogonal in $\mathcal H$. Indeed, in terms of the above notation, these eigenvalues are precisely $\mu_k = \frac{\alpha(k)}{\alpha(1)}$, and the eigenvalues $v_k$ correspond precisely to the spherical harmonics $Y_k$ of degree $k$ under the transformation \eqref{ster proj}. }

Since $\scalprod{v,w}_{\dot{H}^{s}(\R^N)} =  \scalprod{\mathcal L v, w}_{\mathcal H}$, and $[v,w] = \scalprod{v,w}_{\dot{H}^{s}(\R^N)} - \scalprod{v, w}_{\mathcal H}$, they are also pairwise orthogonal with respect to the scalar products $\scalprod{\cdot, \cdot}_{\dot{H}^{s}(\R^N)}$ and $[ \cdot, \cdot]$. In fact, the subspace $T_{0,1}$ is precisely spanned by the first two eigenspaces of $\mathcal L$, while $T_{0,1}^\bot$ is spanned by the higher eigenfunctions. {In particular, we have 
\[ \mathcal L U = \mu_0 U, \qquad \mathcal L \partial_{z_i} U = \mu_1 \partial_{z_i} U, \qquad \mathcal L \partial_{\lambda} U = \mu_1 \partial_{\lambda} U, \]
with $\mu_0 = \frac{1}{p}$ and $\mu_1 = 1$, which can also be checked directly.}

In the same spirit, let us now decompose 
\[ T_{0,1}^\bot = E_0 \oplus E_+ \subset \dot H^s, \]
where $E_0$ is the third eigenspace of $\mathcal L$ (corresponding to spherical harmonics of degree $k = 2$ via \eqref{ster proj}) and $E_+$ is the space spanned by the fourth and higher eigenfunctions (corresponding to spherical harmonics of degree $k \geq 3$). Accordingly, any $\rho \in T_{0,1}^\bot$ can be uniquely written as
\[ \rho = \rho_0 + \rho_+ \in E_0 \oplus E_+. \]
Let $\gamma = \frac{4s}{N+2s+2}$ denote the best constant in Lemma \ref{lm:coercivity}.  {Then we have seen in the previous proof that 
\[ [\rho_0, \rho_0] = \gamma \|\rho_0\|_{\dot{H}^{s}(\R^N)}^2 \qquad \text{ for all } \rho_0 \in E_0. \]
By the same argument as the one used in the proof of Lemma \ref{lm:coercivity}, we can further improve the bound of  \eqref{coerc_ineq_rey} for functions in $E_+$:  there is $\gamma_+ > \gamma$ such that 
\[ [\rho_+, \rho_+] \geq \gamma_+ \|\rho_+\|^2_{\dot{H}^{s}(\R^N)} \qquad \text{ for all } \rho_+ \in E_+. \]
Indeed, using the same arguments as in the proof of Lemma \ref{lm:coercivity}, this is equivalent to the improved spectral inequality 
\[ \scalprod{V, A_s V} \geq \alpha(3) \|V\|_{L^2(\mathbb S^N)}^2, \]
which is valid for functions $V$ in the span of spherical harmonics of degree $k \geq 3$, and $\gamma_+ = 1 - \frac{\alpha(1)}{\alpha(3)} > 1 - \frac{\alpha(1)}{\alpha(2)} = \gamma$.} 
\end{remark}

To prove Theorem \ref{th:main-stability-explicit-constant}, we shall need the following lemma on the existence of optimal $z \in \R^N$, $\lambda>0$ such that $\dist(u, \mathcal M) = \|u - U[z,\lambda]\|_{\dot{H}^{s}(\R^N)}$. 

\begin{lemma}[Optimal $z$, $\lambda$]
\label{lemma optimal z lambda}
Let $\dist(u, \mathcal M) < (1 - \eps) \|U[0,1]\|_{\dot{H}^{s}(\R^N)}$. Then there exist $z \in \R^N$ and $\lambda > 0$ such that $\dist(u, \mathcal M) = \|u - U[z,\lambda]\|_{\dot{H}^{s}(\R^N)}$. Moreover, $u$ is orthogonal in $\dot H^s$ to $\partial_\lambda U[z, \lambda]$ and  $\partial_{z_i} U[z, \lambda]$ ($i = 1,...,N$). 
\end{lemma}

\begin{proof}
We claim that we only need to show that there is $z$, $\lambda$ such that
\begin{equation}
    \label{compactness bound}
    \|u - U[z, \lambda]\|_{\dot{H}^{s}(\R^N)}^2 < \|u\|_{\dot{H}^{s}(\R^N)}^2 + \|U[0,1]\|_{\dot{H}^{s}(\R^N)}^2.
\end{equation}
Indeed, let $\{(z_k, \lambda_k)\}_{k \in \N}$ by a minimizing sequence for $\dist(u, \mathcal M)$. It is straightforward to check that, if either $|z_k| \to \infty$, $\lambda_k \to \infty$, or $\lambda_k \to 0$, then 
\begin{equation}
    \label{dist at infty}
    \lim_{k \to \infty} \|u - U[z_k, \lambda_k]\|_{\dot{H}^{s}(\R^N)}^2 = \|u\|_{\dot{H}^{s}(\R^N)}^2 + \|U[0,1]\|_{\dot{H}^{s}(\R^N)}^2. 
\end{equation}
Thus, by \eqref{compactness bound}, $\{(z_k, \lambda_k)\}_{k \in \N}$ must stay in a compact subset of $\R^N \times (0,\infty)$ and converges thus up to a subsequence. 

Let us now show \eqref{compactness bound}. By assumption there exist $z$ and $\lambda$ such that 
\[ \|u - U[z, \lambda]\|_{\dot{H}^{s}(\R^N)} < \left(1 - \frac{\eps}{2}\right) \|U[0,1]\|_{\dot{H}^{s}(\R^N)}. \]
Moreover, by triangle inequality, 
\[ \|u\|_{\dot{H}^{s}(\R^N)} \geq \|U[z,\lambda]\|_{\dot{H}^{s}(\R^N)} - \|u - U[z,\lambda]\|_{\dot{H}^{s}(\R^N)} > \frac{\eps}{2} \|U[0,1]\|_{\dot{H}^{s}(\R^N)}. \]
Combining these two inequalities, we get 
\begin{align*} \|u\|_{\dot{H}^{s}(\R^N)}^2 + \|U[0,1]\|_{\dot{H}^{s}(\R^N)}^2 &> \left(1 + \frac{\eps^2}{4}\right) \|U[0,1]\|^2 \\ &> \frac{1 + \frac{\eps^2}{4}}{(1 - \frac{\eps}{2})^2} \|u - U[z,\lambda]\|_{\dot{H}^{s}(\R^N)}^2 > \|u - U[z,\lambda]\|_{\dot{H}^{s}(\R^N)}^2. 
\end{align*}
Hence \eqref{compactness bound} holds.

The claimed orthogonality follows from the fact that, by the minimality of $\|u - U[z, \lambda]\|_{\dot{H}^{s}(\R^N)}$, we deduce  
\[ 0 = \frac{\mathrm d}{\mathrm d t}\Bigr\rvert_{t = 0} \|u - U[z+t e_i, \lambda]\|^2_{\dot{H}^{s}(\R^N)} = \frac{\mathrm d}{\mathrm d t}\Bigr\rvert_{t = 0} \|u - U[z, \lambda + t]\|^2_{\dot{H}^{s}(\R^N)}. \]
This completes the proof. 
\end{proof}

\begin{proof}
[Proof of Theorem \ref{th:main-stability-explicit-constant}]
\textbf{Step 1:} \emph{Reducing to $z = 0$, $\lambda = 1$.} By assumptions \eqref{eq:bound-energy}--\eqref{eq:ps-conv} and Lemma \ref{lm:s-struwe}, we may assume without loss of generality that the assumption of Lemma \ref{lemma optimal z lambda} is satisfied for all $k$. For any $u_k$, let $z_k$ and $\lambda_k$ be the parameters given by Lemma \ref{lemma optimal z lambda}. Now consider the transformation 
\begin{equation}
    \label{transformation tilde u k}
    u_k \mapsto \tilde{u}_k(x) := \lambda_k^{-\frac{N-2s}{2}} u_k (z_k + \lambda_k^{-1} x), 
\end{equation} 
which is an isometry of $\dot H^s(\R^N)$. It leaves  $\|(-\Delta)^s u_k - u_k^p\|_{\dot H^{-s}(\R^N)}$ invariant and satisfies 
\[ \dist(\tilde{u}_k, \mathcal M) = \|u - U[0,1]\|_{\dot{H}^{s}(\R^N)} \quad \text{ and } \quad \int_{\R^N} U[z_k, \lambda_k]^{p-1} u_k^2 \dd x = \int_{\R^N} U[0,1]^{p-1} \tilde{u}_k^2 \dd x. \]
Moreover, $\tilde{u}_k$ is orthogonal in $\dot H^s$ to { $\partial_\lambda U[0,1]$ and  $\partial_{z_i} U[0,1]$. }

In summary, by applying the transformation \eqref{transformation tilde u k} and on account of the invariances just described, we may assume in what follows that $\dist(u_k, \mathcal M) = \|u_k - U[0,1]\|_{\dot{H}^{s}(\R^N)}$ for every $k \in \N$. 

\textbf{Step 2:} \emph{Orthogonal decomposition. }
Going forward, we suppress the index $k$ in the sequence $u_k$ and associated parameters to simplify notation. 
Using $\|f\|_{\dot{H}^{-s}(\R^N)} = \|(-\Delta)^{-s/2} f\|_{L^2(\R^N)}$, we write 
\begin{align*}
    \|(-\Delta)^s u - u^p\|^2_{\dot{H}^{-s}(\R^N)} &= \int_{\R^N} ( \Ds u - u^p)  \big((-\Delta)^{-s} ((-\Delta)^s u - u^p)\big) \dd x \\
   &  {= \int_{\R^N} ( \Ds u - u^p)  (  u - (-\Delta)^{-s} u^p ) \dd x } \\
    &= \| u\|_{\dot{H}^s(\R^N)}^2 - 2 \int_{\R^N} u^{p+1} \dd x + \int_{\R^N} u^p(-\Delta)^{-s} u^p \dd x,
    \end{align*}
    where $p = \frac{N+2}{N-2}$. We write $u = \beta U + \rho$ for some $\beta \in \R$ with $\rho \in T^\perp$ and, from Lemma \ref{lm:s-struwe}, we know $\beta \to 1$ {and $\|\rho\|_{\dot{H}^s(\R^N)} \to 0$}. 
    
    Using the orthogonality conditions, we get \begin{align*}
        \|u\|_{\dot{H}^s(\R^N)}^2 &= \beta^2\|U\|_{\dot{H}^s(\R^N)}^2 + \|\rho\|_{\dot{H}^s(\R^N)}^2 , \\
        - 2  \int_{\R^N} u^{p+1} \dd x &= - 2 \beta^{p+1}  \int_{\R^N} U^{p+1} \dd x - (p+1) p \beta^{p-1}\int_{\R^N} U^{p-1} \rho^2  \dd x  + o(\|\rho\|_{\dot{H}^s(\R^N)}^2), \\
        \int_{\R^N} u^p(-\Delta)^{-s} u^p \dd x &= \beta^{2p} \int_{\R^N} U^{p+1} \dd x + p(p-1) \beta^{2p-2} \int_{\R^N} U^{p-1} \rho^2  \dd x \\
        & \quad + p \beta^{2p-2} \int_{\R^N} U^{p-1} \rho (-\Delta)^{-s} p U^{p-1} \rho \dd x  + o\big(\|\rho\|_{\dot{H}^s(\R^N)}^2\big). 
    \end{align*}
{Since $\beta = 1 + o(1)$ and $\int_{\R^N} U^{p-1} \rho^2 \lesssim \|\rho\|^2_{\dot{H}^s(\R^N)}$ by Hölder and Sobolev's inequalities, we moreover have}
 $\beta^{2p-2} \int_{\R^N} U^{p-1} \rho^2  \dd x = \beta^{p-1}  \int_{\R^N} U^{p-1} \rho^2  \dd x + o\big(\|\rho\|_{\dot{H}^s(\R^N)}^2\big)$. 
 Combining all of this, we get 
\begin{align*}
    \|(-\Delta)^s u - u^p\|^2_{\dot{H}^{-s}(\R^N)} &=(\beta-\beta^{p})^2 \|U\|_{\dot{H}^s(\R^N)}^2 +  \|\rho\|_{\dot{H}^s(\R^N)}^2 - 2 p \beta^{p-1} \int_{\R^N} U^{p-1} \rho^2  \dd x \\
    & \quad + \beta^{2p-2} p \int_{\R^N} U^{p-1} \rho (-\Delta)^{-s} p U^{p-1} \rho \dd x  + o\big(\|\rho\|_{\dot{H}^s(\R^N)}^2\big). 
\end{align*}     
 By the orthogonality conditions, we can write, for certain $\alpha_i \in \R$, 
\[ \rho = \sum_{i = 2}^\infty \alpha_i \psi_i, \]
where $(\psi_i)_{i = 0}^\infty$ are eigenfunctions associated to the eigenvalues $\mu_0 = 1/p$, $\mu_1 =  1$, $\mu_2 > 1$, $\mu_3,...$ of the operator $\mathcal L = \frac{\Ds}{p U^{p-1}}$, which are pairwise orthogonal in $\dot{H}^s(\R^N)$ and in  $L^2(\R^N, \, p U^{p-1} \dd x)$ {(see Remark \ref{rk:linearized})}. {We normalize the $\psi_i$ to satisfy $\| \psi_i \|_{\dot{H}^s(\R^N)} = 1$ for every $i$.}

\textbf{Step 3:} \emph{Conclusion of the proof.} Observing the relations $\mu_i = \frac{\|\psi_i\|^2_{\dot{H}^s(\R^N)}}{p\int_{\R^N} U^{p-1} \psi_i^2 \dd x}$ and $(-\Delta)^{-s} (p U^{p-1} \psi_i) = \mu_i^{-1} \psi_i$, {by the $\dot{H}^s(\R^N)$-orthonormality of the $\psi_i$}, we have
\begin{align*}
 \|\rho\|_{\dot{H}^s(\R^N)}^2 &= \sum_{i=2}^\infty \alpha_i^2 \|\psi_i\|_{\dot{H}^s(\R^N)}^2, \\
 p \int_{\R^N} U^{p-1} \rho^2  \dd x &= \sum_{i=2}^\infty \alpha_i^2 p \int_{\R^N} U^{p-1} \psi_i^2  \dd x = \sum_{i=2}^\infty \alpha_i^2 \mu_i^{-1} \|\psi_i\|_{\dot{H}^s(\R^N)}^2, \\
  p \int_{\R^N} U^{p-1} \rho (-\Delta)^{-s} p U^{p-1} \rho \dd x &= \sum_{i=2}^\infty \alpha_i^2 \mu_i^{-1} p \int_{\R^N} U^{p-1} \psi_i^2  \dd x = \sum_{i=2}^\infty \alpha_i^2 \mu_i^{-2} \|\psi_i\|_{\dot{H}^s(\R^N)}^2.
\end{align*}
Again putting everything together, we find 
\begin{align}
    \|(-\Delta)^s u - u^p\|^2_{\dot{H}^{-s}(\R^N)}&= \sum_{i=2}^\infty \alpha_i^2 (1 - 2 \beta^{p-1} \mu_i^{-1} + \beta^{2p-2} \mu_i^{-2}) \|\psi_i\|_{\dot{H}^s(\R^N)}^2 +  o\big(\|\rho\|_{\dot{H}^s(\R^N)}^2\big) \nonumber \\
    & = \sum_{i=2}^\infty \alpha_i^2 (1 - \mu_i^{-1})^2 \|\psi_i\|_{\dot{H}^s(\R^N)}^2 +  o\big(\|\rho\|_{\dot{H}^s(\R^N)}^2\big) \nonumber \\
    & \geq ((1 - \mu_2^{-1})^2 + o(1)) \|\rho\|_{\dot{H}^s(\R^N)}^2 = (\gamma_{N,s}^2 + o(1)) \|\rho\|_{\dot{H}^s(\R^N)}^2.  \label{spectral estimate}
\end{align}
{Since $(1 - \mu_i^{-1})$ form a strictly increasing sequence, asymptotic equality holds here if and only if $u_k$ satisfies \eqref{optimality condition 2 thm}. For readability, we drop in the following the error term corresponding to \eqref{spectral estimate} and focus on explicitly checking the other, somewhat less standard, optimality condition \eqref{optimality condition 1 thm}.}

On the other hand, the denominator of the quotient we want to evaluate is 
\[ \dist(u, \mathcal M)^2 = \|u - U\|_{\dot{H}^s(\R^N)}^2 = (1 - \beta)^2 \|U\|_{\dot{H}^s(\R^N)}^2 + \|\rho\|_{\dot{H}^s(\R^N)}^2. \]
Hence the full quotient can be estimated by 
\begin{equation}
    \label{quotient proof}
    \frac{ \|(-\Delta)^s u - u^p\|^2_{\dot{H}^{-s}(\R^N)}}{\dist(u, \mathcal M)^2} \geq \frac{(\beta-\beta^{p})^2 \|U\|_{\dot{H}^s(\R^N)}^2 + (\gamma_{N,s}^2 + o(1)) \|\rho\|_{\dot{H}^s(\R^N)}^2}{(1 - \beta)^2 \|U\|_{\dot{H}^s(\R^N)}^2 + \|\rho\|_{\dot{H}^s(\R^N)}^2} = {\frac{a + b }{A + B}, }
\end{equation} 
{where we have defined 
\begin{equation}
    \label{a A definition}
    a = a_k :=  (\beta-\beta^{p})^2  \|\rho\|_{\dot{H}^s(\R^N)}^{-2} \|U\|_{\dot{H}^s(\R^N)}^2, \qquad A = A_k:= (1 - \beta)^2 \|\rho\|_{\dot{H}^s(\R^N)}^{-2}  \|U\|_{\dot{H}^s(\R^N)}^2 ,
\end{equation} 
 $b = b_k:= \gamma_{N,s}^2 + o(1)$, and $B := 1$. Let us also denote $a_\infty = \lim_{k \to \infty} a_k$, $A_\infty = \lim_{k \to \infty} A_k$, and $b_\infty = \lim_{k \to \infty} b_k = \gamma_{N,s}^2$. }

{By a Taylor expansion and the fact that $\beta = 1 + o(1)$, we have $(1- \beta^{p-1}) = (p-1 + o(1))(1 - \beta)$. Hence 
\begin{equation}
    \label{a/A}
     \frac{a}{A} = \frac{(\beta-\beta^{p})^2 }{(1 - \beta)^2} = \frac{\beta^2 (1-\beta^{p-1})^2 }{(1 - \beta)^2} =  (p-1)^2 + o(1). 
\end{equation}
Because of the strict inequality 
\begin{equation}
    \label{strict ineq proof}
    p-1 = \frac{4s}{N-2s} > \frac{4s}{N+2s+2} +o(1) = \gamma_{N,s} +o(1),
\end{equation} 
we thus see $\frac{a_k}{A_k} > \frac{b_k}{B}$ (whenever $A_k \neq 0$). From this, we conclude that 
\begin{equation}
    \label{final estimate proof}
    \frac{ \|(-\Delta)^s u - u^p\|^2_{\dot{H}^{-s}(\R^N)}}{\dist(u, \mathcal M)^2}  = \frac{a + b}{A + B} \geq \frac{b}{B} = \gamma_{N,s}^2 + o(1),
\end{equation} 
which proves \eqref{asymptotic inequality thm}. It remains to discuss the cases of equality in this estimate. 
If $A_\infty > 0$, then, by \eqref{a/A}, we also have $a_\infty > 0$ and, by \eqref{strict ineq proof}, still $\frac{a_\infty}{A_\infty} > \frac{b_\infty}{B}$. Then 
\[ \lim_{k \to \infty}  \frac{ \|(-\Delta)^s u_k - u_k^p\|^2_{\dot{H}^{-s}(\R^N)}}{\dist(u_k, \mathcal M)^2} = \frac{a_\infty + b_\infty}{A_\infty + B} > \frac{b_\infty}{B} = \gamma_{N,s}^2.  \]
Conversely, if $A_\infty = 0$, then also $a_\infty = 0$ by \eqref{a/A} and equality in \eqref{final estimate proof} holds.}

{In summary, equality in \eqref{final estimate proof} holds if and only if $A_\infty = 0$, which, in view of \eqref{a A definition}, is equivalent to $|1- \beta| = o(\|\rho\|_{\dot{H}^s(\R^N)})$. Thus \eqref{optimality condition 1 thm} is proved. } 
\end{proof}

Let us conclude this section by proving 
Theorem \ref{th:main-stability1}. 

\begin{proof}[Proof of Theorem \ref{th:main-stability1}]

\textbf{Step 1:} \emph{Proof of the inequality \eqref{error bound thm}.} By Theorem \ref{th:main-stability-explicit-constant}, there exists $\delta_0 > 0$ such that, whenever $\delta(u) := \|(-\Delta)^s u - u^{2^*_s-1}\|_{\dot{H}^{-s}(\R^N)} \leq \delta_0$, we have  
\[ \|u - U[z,\lambda] \|_{\dot{H}^{s}(\R^N)} \leq \frac{2 }{\gamma_{N,s}} \delta(u) \]
for some $z \in \R^N$ and $\lambda > 0$ as desired. On the other hand, we may observe that, by \eqref{eq:bound-energy}, 
\[ \|u - U[0,1]\|_{\dot{H}^{s}(\R^N)} \leq \|u\|_{\dot{H}^{s}(\R^N)} + \|U[0,1]\|_{\dot{H}^{s}(\R^N)} \leq 3 S_{N,s}^{N/4s} < \infty. \]
Thus, if $\delta(u) > \delta_0$, then the bound $\|u - U[z,\lambda] \|_{\dot{H}^{s}(\R^N)} \leq C \delta(u)$ holds trivially if we take $C:= 3 S_{N,s}^{N/4s}  \delta_0^{-1}$. 

\textbf{Step 2:} \emph{Strictness of the inequality.} To prove that the strict inequality $c^1_{CP}(s) < \gamma_{N,s}^2$ holds, we compute the expansion of the quotient, in the limit as   \(\varepsilon \rightarrow 0\), for a test function \(u= U+\varepsilon \rho\), with \(U=U[0,1]\) and \(\rho \in E_0\) a third eigenfunction of the linearized operator $\mathcal L$ (see Remark \ref{rk:linearized} for the notation). The denominator of the quotient is given by 
\begin{align*}
\operatorname{dist}(u, \mathcal{M})^2=\varepsilon^2\|\rho\|_{\dot{H}^s(\R^N)}^2.
\end{align*}
For the numerator, we do similar computations as in the proof of Theorem \ref{th:main-stability-explicit-constant}, but taking one order more into account in the Taylor expansions of \((U+\varepsilon \rho)^{p+1}\) and \((U+\varepsilon \rho)^p\) (i.e., expanding up to the third order).
Recalling that \((-\Delta)^{-s} U^p=U\) and \((-\Delta)^{-s}\left(p U^{p-1} \rho\right)=\mu_3^{-1} \rho\), we get
\begin{align*}
\|u\|_{\dot{H}^s(\R^N)}^2&=\|U\|_{\dot{H}^s(\R^N)}^2+\varepsilon^2\|\rho\|_{\dot{H}^s(\R^N)}^2,\\
-2 \int_{\mathbb{R}^N} u^{p+1} \dd x & =-2 \int_{\mathbb{R}^N} U^{p+1} \dd  x-(p+1) p \varepsilon^2 \int_{\mathbb{R}^N} U^{p-1} \rho^2 \dd x \\
&\qquad -\frac{(p+1) p(p-1)}{3} \varepsilon^3 \int_{\mathbb{R}^N} U^{p-2} \rho^3 \dd x+o\left(\varepsilon^3\right), \\
\int_{\mathbb{R}^N} u^p(-\Delta)^{-s} u^p \dd x&= \int_{\mathbb{R}^N} U^{p+1} \dd x+p(p-1) \int_{\mathbb{R}^N} U^{p-1} \rho^2 \dd x +p \mu_3^{-1} \int_{\mathbb{R}^N} U^{p-1} \rho^2 \dd x \\
& \qquad + \varepsilon^3\left(\frac{p(p-1)(p-2)}{3}+p(p-1) \mu_3^{-1}\right) \int_{\mathbb{R}^N} U^{p-2} \rho^3 \dd x+o\left(\varepsilon^3\right) .
\end{align*}
Putting these expansions together and using the fact that \(p \int_{\mathbb{R}^N} U^{p-1} \rho^2 \dd x=\mu_3^{-1}\), we get 
\begin{align*} 
\left\|(-\Delta)^s u-u^p\right\|_{H^{-s}(\R^N)}^2=\varepsilon^2\|\rho\|_{\dot{H}^s(\R^N)}^2\left(1-2 \mu_3^{-1}+\mu_3^{-2}\right)+\varepsilon^3 p(p-1)\left(\mu_3^{-1}-1\right) \int_{\mathbb{R}^N} U^{p-2} \rho^3 \dd x+o\left(\varepsilon^3\right).
\end{align*} 

In conclusion, the quotient equals
\begin{align*}
\frac{\left\|(-\Delta)^s u-u^p\right\|_{H^{-s}(\R^N)}^2}{\operatorname{dist}(u, \mathcal{M})^2} &=\left(1-\mu_3^{-1}\right)^2+\varepsilon^3 p(p-1)\left(\mu_3^{-1}-1\right) \int_{\mathbb{R}^N} U^{p-2} \rho^3 \dd x + o(\eps^3)\\
&=\gamma_{N, s}^2+\varepsilon^3 p(p-1)\left(\mu_3^{-1}-1\right) \int_{\mathbb{R}^N} U^{p-2} \rho^3 \dd x+o\left(\varepsilon^3\right).
\end{align*}

Since \(\mu_3>1\), the coefficient \(p(p-1)\left(\mu_3^{-1}-1\right)\) of the \(\varepsilon^3\) term is strictly negative. Therefore, we can conclude, {for $\eps > 0$ small enough}, if we are able to choose \(\rho \in E_0\) such that \(\int_{\mathbb{R}^N} U^{p-2} \rho^3 \dd x>0\). That this is possible has been observed in \cite[p. 5]{2210.08482}. For the sake of completeness, we reproduce the argument here: the choice is 
\begin{align*}
\rho(x)=\left(\frac{2}{1+|x|^2}\right)^{\frac{N}{p+1}} v(\mathcal{S}(x)),
\end{align*}
where \(\mathcal{S}: \mathbb{R}^N \rightarrow \mathbb{S}^N\) is the (inverse) stereographic projection and \(v\) is given by
\begin{align}
\label{test function v}
v(\omega)=\omega_1 \omega_2+\omega_2 \omega_3+\omega_3 \omega_1 .
\end{align}
This definition is what necessitates the extra assumption $N \geq 2$. The function \(v\) is a spherical harmonic of degree \(\ell=2\), so it corresponds to a third eigenfunction of the operator $\mathcal L$. Then, passing to the sphere, we compute
\begin{align*}
\int_{\mathbb{R}^N} U^{2^*_s-3} \rho^3 \dd x=2^{\frac{3(N-2 s)}{2}-N} \int_{\mathbb S^N} v(\omega)^3 \dd \omega=6 \times 2^{\frac{3(N-2 s)}{2}-N} \int_{\mathbb S^N} \omega_1^2 \omega_2^2 \omega_3^2 \dd \omega>0,
\end{align*}
because all the other contributions from multiplying out \(v(\omega)^3\) cancel by oddness. This concludes the proof.
\end{proof}

\section{Asymptotics for the fractional fast diffusion equation}
\label{sec:yamabe}

As an application of Theorem \ref{th:main-stability1}, we are able to prove Theorem \ref{th:main-fde}.

\begin{proof}[Proof of Theorem \ref{th:main-fde}]

\textbf{Step 1:} \emph{Change of variables and preliminaries.}  By a scaling argument, we can consider $z = 0$ and $\lambda = 1$. We change variables by defining $w: (0, \infty) \times \R^N \to [0, \infty)$ to be
\begin{align}
\label{w in terms of u}
    w(\tau, x) := \left(\frac{u(t,x)}{U_{\bar T, 0,1}(t,x)}\right)^{1/p} U[0,1](x) =\left(\frac{u(t,x)}{\left(\frac{p-1}{p}\right)^{\frac{p}{p-1}}(\bar T-t)^{\frac{p}{p-1}} }\right)^{1/p} ,
\end{align}
for $t = \bar T\left(1-\exp\left(\frac{-(p-1)\tau}{p}\right)\right)$, i.e. $\bar T-t = \bar T\exp\left(\frac{-(p-1)\tau}{p}\right)$. 

We claim that the new function $w$ then solves the following equation: 
\begin{align}\label{eq:wp}
    \partial_\tau w^p + (-\Delta)^s w = w^{p}.
\end{align}
{Indeed, set 
\[ \tilde w(\tau, x) = (\bar T - \tilde{t}(\tau))^{-\frac{1}{p-1}} u(\tilde{t}(\tau),x)^\frac{1}{p},  \]
with $\tilde{t}(\tau) = \bar T (1 - e^{-\tau})$. Observing that $\tilde{t}'(\tau) = e^{-\tau} = \bar T - \tilde{t}(\tau)$, a direct computation shows 
\[ \partial_\tau (\tilde{w}(\tau,x)^p) = \frac{p}{p-1} \tilde{w}(\tau,x)^p + (\bar T - \tilde{t}(\tau))^{-\frac{1}{p-1}} (\partial_t u)(\tilde{t}(\tau), x) \]
From this, using the equation for $u$, we get 
\begin{align*}
    (-\Delta)^s \tilde{w} &= (\bar T - \tilde{t}(\tau))^{-\frac{1}{p-1}} ((-\Delta)^s u^\frac{1}{p})(\tilde{t}(\tau), x) = - (\bar T - \tilde{t}(\tau))^{-\frac{1}{p-1}}  (\partial_t u)(\tilde{t}(\tau), x) \\
    &= - \partial_\tau (\tilde{w}(\tau, x)^p) + \frac{p}{p-1} \tilde{w}(\tau, x)^p.
\end{align*}
Since $w(\tau, x) = \left(\frac{p-1}{p} \right)^{-\frac{1}{p-1}} \tilde w(\frac{p}{p-1} \tau, x)$, this  implies \eqref{eq:wp}.}

{As a consequence of the regularity assumption on $u_0$ and \cite[Theorem 5.1]{MR3276166}, we have that $w, \tilde{w} \in C^\infty((0, \infty) \times \R^N)$ and, correspondingly, $u \in C^\infty( (0, T) \times \R^N)$. Indeed, our function $\tilde{w}$ is related to the function $v$ from \cite{MR3276166} by $\tilde{w}(x) = \left(\frac{2}{1+|x|^2}\right)^{\frac{N}{p+1}} v(\mathcal S(x))$, where $\mathcal S: \R^N \to \mathbb S^N$ is (inverse) stereographic projection. Since $\mathcal S$ is smooth, the claimed smoothness of $\tilde{w}$, and hence of $w$ follows from the fact that $v \in C^\infty((0, \infty) \times \mathbb S^N)$, which is proved in \cite[Theorem 5.1]{MR3276166}.}

Moreover, Lemma \ref{lm:s-pino} yields 
\begin{align}\label{eq:wconvfd}
    w(\tau) \to U[0,1] \quad \text{ in  $L^{2^*_s}(\R^N)$ as $\tau \to \infty$.}
 \end{align}

\textbf{Step 2:} \emph{Upgrading to $\dot H^s$-convergence.} In this step, we will prove that the convergence \eqref{eq:wconvfd} implies in fact the slightly better 
\begin{align}
\label{eq:wconvfdHs}
    w(\tau) \to U[0,1] \quad \text{ in  $\dot H^s (\R^N)$ as $\tau \to \infty$.}
 \end{align}

Let us introduce the functional
$$J(w) := \frac{1}{2}\int_{\R^N} |(-\Delta)^{s/2} w|^2 \dd x - \frac{1}{2^*_s} \int_{\R^N} w^{2^*_s} \dd x.$$
We note that $$J(U[z,\lambda]) = S_{N,s}\left(\tfrac{1}{2}-\tfrac{1}{2_s^*}\right)$$ 
and, using \eqref{eq:wconvfd} and \eqref{eq:wp}, we obtain  
\begin{align}
    \label{eq:J-dec} &\tfrac{\mathrm d}{\mathrm d \tau} J(w(\tau, \cdot)) = -\mathcal R,  \\
    \label{eq:J-nn} & \liminf_{\tau \to \infty} J(w(\tau, \cdot)) \ge -\frac{1}{2^*_s} \int_{\R^N} U[0,1]^{2^*_s} \dd x.
\end{align}
where the dissipation rate is given by 
\begin{align}\label{eq:diss-fd}
    \mathcal R :=  \frac{1}{p} \int_{\R^N} \frac{((-\Delta)^sw(\tau,x)-w^p(\tau,x))^2}{w^{p-1}(\tau,x)} \dd x \ge 0
\end{align}
(cf. the computations in \cite[Lemma 5.3 \& 5.4]{MR3276166} for further details). In particular, the quantity $\tau \mapsto J (w(\tau,\cdot))$ is decreasing.

Let us denote 
\begin{equation}
    \label{delta definition}
    \delta(\tau) := \delta(w(\tau)):= \| (-\Delta)^s w(\tau) - w(\tau)^p \|^2_{L^{(2_s^*)'}(\R^N)}. 
\end{equation}
Using H\"older's inequality, we further estimate the dissipation rate as follows: 
\begin{align}
     \mathcal R := \mathcal R(\tau) &:=   \frac{1}{p} \int_{\R^N} \frac{((-\Delta)^s w(\tau,x)+w^p(\tau,x))^2}{w^{p-1}(\tau,x)} \dd x \nonumber \\
     &\ge \frac{1}{p} \delta(w)^2 \left(\int_{\R^N} w(\tau)^{2^*_s} \dd x \right)^{-\frac{2s}{N}} \label{J gronwall bound easy} \\
         & =  \left(\frac{1}{p S_{N,s}} + o(1)\right)  \delta(w)^2  \nonumber
\end{align}
because  $ \int_{\R^N} w(\tau, x) ^{2^*_s} \dd x \to \int_{\R^N} U[0,1] ^{2^*_s} \dd x = S_{N,s}^\frac{N}{2s}$ as $\tau \to \infty$. Here and in the rest of the proof, $o(1)$  denotes a quantity that tends to zero as $\tau \to \infty$.

The last estimate gives $\frac{\dd }{\dd \tau} J(w) \leq -c \delta(w)^2$, for some $c > 0$. Since on the other hand $J(w)$ is bounded from below by \eqref{eq:J-nn}, we must have $\delta(w(\tau_i)) \to 0$ along a sequence $\tau_i \to \infty$. 
Moreover, 
\begin{equation}
    \label{sobolev bound dual}
    \|(-\Delta)^s w - w^p\|_{\dot H^{-s}(\R^N)} \leq S_{N,s}^{-1/2} \delta(w)
\end{equation}
by duality, and $\|w(\tau)\|_{\dot H^s(\R^N)}$ is bounded uniformly in $\tau$ because $J(w(\tau))$ is decreasing. Thus Lemma \ref{lm:s-struwe} applies (or strictly speaking, its generalization replacing hypothesis \eqref{eq:bound-energy} by mere $\dot H^s$ boundedness; see again \cite[Theorem 1.1]{MR3316602}) to yield $\dot{H}^s$-convergence:
\begin{align}\label{eq:hconvfd}
    w(\tau_i) \to U[0,1] \quad \text{ in  $\dot{H}^s(\R^N)$ as $i \to \infty$.}
 \end{align}
 Since $J(w(\tau))$ is decreasing,  
 \begin{align}
 \label{J w tau i limit}
   \lim_{\tau \to \infty} J(w(\tau)) =   \lim_{i \to \infty} J(w(\tau_i,\cdot)) = J(U[0,1]),
 \end{align}
which, together with \eqref{eq:wconvfd}, yields
\begin{align}\label{eq:convsn}
    \int_{\R^N} |(-\Delta)^s w(\tau,x)|^2  \dd x \to   \int_{\R^N} |(-\Delta)^s U[0,1]|^2 \dd x = S^\frac{N}{2s}_{N,s} \qquad \text{ as } \tau \to \infty.
\end{align}
From $\dot H^s$ boundedness of $w(\tau)$ and \eqref{eq:wconvfd}, it is easy to see that $w(\tau) \rightharpoonup U[0,1]$ weakly in $\dot H^s$ as $\tau \to \infty$. Together with the convergence of the $\dot{H}^s$ norms from \eqref{eq:convsn}, this yields \eqref{eq:wconvfdHs}, as desired. 

\textbf{Step 3:} \emph{Upgrading to exponential $L^{2^*_s}$-convergence.} This is the central and conceptually most involved step of the proof, into which Theorem \ref{th:main-stability-explicit-constant} enters. Our goal is to improve the qualitative statement of \eqref{eq:wconvfd} to the exponential convergence 
\[ \|w(\tau) - U[0,1] \|_{L^{2^*_s}(\R^N)} \leq C e^{-\kappa \tau}, \]
for all $\kappa < \kappa_2(N,s)$, for some explicit number $\kappa_2(N,s) > 0$ given in \eqref{kappa 2} below and a generic constant $C > 0$ which may change from line to line and is independent of $\tau$.

Thanks to \eqref{eq:wconvfdHs}, we have $\|w(\tau)\|_{\dot H^s(\R^N)} \to \|U\|_{\dot H^s(\R^N)} = S_{N,s}^\frac{2s}{N}$. Moreover, since $J(w(\tau))$ is decreasing and convergent, we have $\mathcal R \to 0$ as $\tau \to \infty$ by \eqref{eq:J-dec}, hence $\|(-\Delta)^s w(\tau) - w(\tau)^p\|_{\dot H^{-s}(\R^N)} \to 0$ by \eqref{J gronwall bound easy} and \eqref{sobolev bound dual}. 
Thus the sequence $w(\tau)$ satisfies the hypotheses of Theorem \ref{th:main-stability-explicit-constant}, as $\tau \to \infty$. We conclude by that theorem that we can write $w(\tau) = U(\tau)  + \rho(\tau)$, where $U(\tau) = U[z(\tau), \lambda(\tau)]$ is a Talenti bubble with $z(\tau) \to 0$ and $\lambda(\tau) \to 1$, and $\rho(\tau)$ satisfies the bound
\begin{equation}
    \label{rho bound fde}
    \begin{split}
   \|\rho(\tau)\|^2_{\dot H^s(\R^N)} &\leq (\kappa_1(N,s) +o(1)) \|(-\Delta)^s w(\tau) - w(\tau)^p\|_{\dot{H}^{-s}(\R^N) }^2  
   \\ &\leq \frac{\kappa_1(N,s) + o(1)}{S_{N,s}} \delta(w(\tau))^2.
   \end{split}
\end{equation} 
Here, the constant $\kappa_1(N,s):= \gamma_{N,s}^{-2} = \left(\frac{N+2s+2}{4s}\right)^2$ is obtained from the best constant $\gamma_{N,s}$ in Theorem \ref{th:main-stability-explicit-constant}. 

Expanding $J(U(\tau) + \rho(\tau))$ and using $(-\Delta)^s U(\tau) = U(\tau)^p$, we obtain 
\begin{align*}
    J(w(\tau)) &= J(U(\tau)) + \frac 12 \|\rho(\tau)\|^2_{\dot H^s(\R^N)} - \frac{1}{2_s^*} \int_{\R^N} \left( (U(\tau) + \rho(\tau))^{2^*_s} - U(\tau)^{2^*_s} - 2^*_s U(\tau)^p \rho(\tau) \right) \dd x.
\end{align*}
Since $2^*_s > 1$ and $w \geq 0$, the integrand on the right side is pointwise positive by Bernoullli's inequality. Observing $J(U(\tau)) = J(U[0,1])$ and using \eqref{rho bound fde}  and \eqref{sobolev bound dual}, we obtain 
\begin{equation}
    \label{J gronwall bound hard}
    J(w(\tau)) - J(U[0,1]) \leq \frac{\kappa_1(N,s) + o(1)}{S_{N,s}} \delta(w(\tau))^2. 
\end{equation}

Summarizing, the content of \eqref{eq:J-dec}, \eqref{J gronwall bound easy}, and \eqref{J gronwall bound hard} can be rephrased as 
\begin{equation}
    \label{J gronwall bound full}
    \begin{split} 
    \frac{\mathrm d}{\mathrm d \tau} \left(J(w(\tau)) - J(U[0,1]) \right) &\leq - \left(\frac{1}{p S_{N,s}} + o(1)\right) \delta(\tau)^2 
    \\ &\leq -\left(\frac{1}{\kappa_1(N,s) p} + o(1)\right) \left(J(w(\tau)) - J(U[0,1]) \right).
    \end{split}
\end{equation}
Together with the fact that $J(w(\tau))$ is decreasing in $\tau$ and \eqref{J w tau i limit}, this implies 
\[ 0 \leq J(w(\tau)) - J(U[0,1]) \leq C \exp( -\kappa \tau)   \qquad \text{ for every } \kappa < \frac{1}{\kappa_1(N,s) p}.   \]
Integrating the first inequality of \eqref{J gronwall bound full} from $\tau$ to $\infty$, this exponential bound gives
\[ \int_\tau^\infty \delta(t)^2 \dd t \leq C \exp( -\kappa \tau)   \qquad \text{ for every } \kappa < \frac{1}{\kappa_1(N,s) p}. \]
By  Cauchy-Schwarz' inequality, we can also turn this into 
\begin{align*}
     \int_\tau^\infty \delta(t) \dd t &= \sum_{k = 0}^\infty \int_{\tau + k}^{\tau + k +1} \delta(t) \dd t \leq \sum_{k = 0}^\infty \left( \int_{\tau + k}^{\tau + k +1} \delta(t)^2 \dd t \right)^{1/2}
     \\ &\leq C  \sum_{k = 0}^\infty \exp(- \kappa (\tau + k)) \leq C \exp (- \kappa \tau) 
\end{align*}
for every 
\begin{equation}
    \label{kappa 2} \kappa < \kappa_2(N,s) := \frac{1}{2 \kappa_1(N,s) p}. 
\end{equation}
Now the equation \eqref{eq:wp} satisfied by $w(\tau)$ implies 
\[ w(\tau)^p - w(t)^p = \int_\tau^t ((-\Delta)^s w - w^p)(\sigma) \dd \sigma. \]
Taking $L^{(2^*_s)'}$ norms gives 
\begin{align}
\| w(\tau) - w(t)\|_{L^{2^*_s}(\R^N)}& \leq C \|w(\tau)^p - w(t)^p\|_{L^{(2_s^*)'}} \leq C \int_\tau^t \|((-\Delta)^s w - w^p)(\sigma)\|_{L^{(2_s^*)'}(\R^N)} \dd \sigma \\
& \leq C \int_\tau^\infty \delta(\sigma) \dd \sigma \leq C  \exp (- \kappa \tau)
\end{align}
for every $\kappa < \kappa_2(N,s)$. By letting $t \to \infty$, we deduce that 
\begin{equation}
\label{w expconv Lp final}
    \| w(\tau) - U[0,1]\|_{L^{2^*_s}(\R^N)} \leq C \exp (-\kappa \tau)  \qquad \text{ for every } \kappa < \kappa_2(N,s) . 
\end{equation} 

\textbf{Step 4:} \emph{Upgrading to $L^\infty$-convergence. } 
To conclude the proof of the theorem, we show that
\begin{equation}
\label{w tau proof final} 
\left\| \frac{w(\tau)}{U[0,1]} - 1 \right\|_{L^\infty(\R^N)} \leq \exp(-\kappa \tau), \qquad \text{ for all } \kappa < \frac{2}{N + 2 - 2s} \kappa_2(N,s) {=:  \kappa_3(N,s)}, \end{equation}
Indeed, undoing the change of variables \eqref{w in terms of u}, we then deduce 
\[ \left\|\frac{u(t)}{U_{T,0,1}} -1\right\|_{L^\infty(\R^N)} = \left\|\frac{w(\tau)}{U[0,1]} -1\right\|_{L^\infty(\R^N)} \leq C \exp(\kappa \tau)\leq C (T-t)^{\kappa \frac{p}{p-1}}
\]
for all $\kappa < \kappa_3(N,s)$. 
Since a computation shows that $\frac{p}{p-1} \kappa_3(N,s)$ is equal to the constant $\kappa_{N,s}$ given in Theorem \ref{th:main-fde}, the bound \eqref{eq:asy} follows. 

To deduce \eqref{w tau proof final} from \eqref{w expconv Lp final}, it is convenient to change variables once more and pass to the unit sphere $\mathbb S^N$ via stereographic projection. Indeed, let $\mathcal S: \R^N \to \mathbb S^N \subset \R^{N+1}$ be defined by $\mathcal S(x) = \left(\frac{2x}{1+|x|^2}, \frac{|x|^2 - 1}{1 + |x|^2}\right)$ and define a function $v$ on $[0, \infty) \times \mathbb S^N$ by setting 
\[ v(\tau, \mathcal S(x)) = \frac{w(\tau, x)}{U[0,1](x)}. \]
Observing that the Jacobian of $\mathcal S$ is proportional to $U[0,1]^{2^*_s}$, estimate \eqref{w expconv Lp final} becomes
\[ \|v(\tau) - 1\|_{L^{2^*_s}(\mathbb S^N)} \leq C \exp (-\kappa \tau)  \qquad \text{ for every } \kappa < \kappa_2(N,s). \]
Moreover, by \cite[Proof of Theorem 3.1 \& Proposition 5.2]{MR3276166}, $v(\tau)$ is uniformly Lipschitz on $\mathbb S^N$. Since moreover $\|v(\tau) - 1\|_{L^\infty(\mathbb S^N)}$ as a consequence of Lemma \ref{lm:s-pino}, the interpolation-type inequality from Lemma \ref{lemma lipschitz interpolation} below applies to give \eqref{w tau proof final} as desired. 
\end{proof}

\begin{remark}[On the sharpness of the estimate] 
\label{remark sharpness}
Let us complement the preceding proof with a heuristic argument giving a candidate for the truly optimal exponential decay rate. The linearization of \eqref{eq:wp} near the stationary solutions $U = U[0,1]$ is 
\[ p U^{p-1} \partial_\tau h + (-\Delta)^s h = p U^{p-1} h \]
or, equivalently,
\begin{equation}
    \label{h equation}
    \partial_\tau h + (\mathcal L - 1)h = 0 
\end{equation} 
with $\mathcal L$ as in Remark \ref{rk:linearized}. Then a solution $w(\tau)$ to \eqref{eq:wp} that converges to $U$ should be well approximated by $w \sim U + h$, where $h$ solves \eqref{h equation} and tends to zero as $\tau \to \infty$. 

In view of this, the slowest possible exponential convergence rate should be given by $\mu_3 - 1 = \frac{1}{1-\gamma_{N,s}} - 1 = \frac{4s}{N-2s +2} > 0$. Indeed, since $\mu_1 < 1$ and $\mu_2 = 1$ by Remark \ref{rk:linearized}, $\mu_3$ is the smallest eigenvalue of $\mathcal L$ such that $\mu_i - 1$ is positive.  



A computation shows, however, that $\mu_3 - 1 > \kappa_3(N,s)$, so the constant from \eqref{w tau proof final} does not appear to be sharp. In fact, one can check that one even has  $\mu_3 - 1 > \frac{1}{\kappa_1(N,s) p}$, which is the constant we obtain at the energy level for the convergence rate of $J(w(\tau)) - J(U)$. 
\end{remark}

We conclude by proving the interpolation inequality used in the proof of Theorem \ref{th:main-fde}. 

\begin{lemma}[Interpolation inequality]
\label{lemma lipschitz interpolation}
For every $M \geq 1$, there is $\delta_0 > 0$ such that the following holds. If $\|f\|_{L^\infty(\mathbb S^N)} \leq \delta_0$ and $f$ is Lipschitz with constant $\operatorname{Lip}(f) \leq M$, then there is $C > 0$ only depending on $N$, $s$, and $M$ such that 
\[ \|f\|_{L^\infty(\mathbb S^N)} \leq C \|f\|_{L^{2^*_s}(\mathbb S^N)}^\frac{2}{N + 2 - 2s}. \]
\end{lemma}

\begin{proof}
Denote $\delta := \|f\|_{L^{\infty}(\mathbb S^N)}$ and let $\xi_0$ be a point with $f(\xi_0) = \delta$. For $\delta_0>0$ such that $\delta_0/M$ is small enough, the geodesic ball $B = B(\xi_0, \delta/M)$ is well-defined {for every $\delta \leq \delta_0$} and we have 
\[ f(\xi) \geq \delta - 2 M |\xi - \xi_0| \qquad \text{ for every } \xi \in B.  \]
This implies, with $\mathrm d S(\xi)$ denoting the usual surface measure of $\mathbb S^N$,
\begin{align*}
    \int_{\mathbb S^N} f^{2^*_s}(\xi) \dd S(\xi) &\geq \int_B f^{2^*_s} \dd S(\xi) \geq {\frac{1}{2}} \int_{\{|x| \leq { \frac{\delta}{8M} }\}} (\delta - {4} M |x|)^{2^*_s} \dd x \\
    &= \delta^{N + 2^*_s} M^{-N} \int_{\{|x| \leq 1/8\}}(1 - 2 |x|)^{2^*_s} \dd x. 
\end{align*} 
{Here, we used stereographic projection which sends $\xi_0$ to $0 \in \R^N$ in order to pass to the Euclidean space. The fact that, for $\delta_0$ small enough, the stereographic projection is close to an isometry on $B(\xi_0, \delta_0/M)$  justifies the estimate we applied.} 
This chain of inequalities implies the claim. 
\end{proof}

\vspace{3mm}

\section*{Acknowledgments}

N. De Nitti is a member of the Gruppo Nazionale per l’Analisi Matematica, la Probabilità e le loro Applicazioni (GNAMPA) of the Istituto Nazionale di Alta Matematica (INdAM). He has been partially supported by the Alexander von Humboldt Foundation and by the TRR-154 project of the Deutsche Forschungsgemeinschaft (DFG). 

T. König acknowledges partial support through ANR BLADE-JC ANR-18-CE40-002. 

We thank F. Glaudo and T. Jin for helpful discussions on the topic of this work.  {Finally, we are grateful to the anonymous referee, whose comments improved the presentation of the paper.}

\vspace{3mm}

\bibliographystyle{abbrv}
\bibliography{FracSobolevQ-ref.bib}

\vfill

\end{document}